\documentclass[fleqn,reqno,11pt,final]{amsart}
\usepackage[utf8]{inputenc}
\usepackage[T1]{fontenc}
\usepackage{etex}
\usepackage[a4paper,twoside=false]{geometry}
\usepackage{lmodern}
\usepackage{mathrsfs}
\usepackage{pifont}
\usepackage{array}
\usepackage{microtype}
\usepackage{enumitem}
\usepackage{amsmath}
\usepackage{amsfonts}
\usepackage{amssymb}
\usepackage{textcomp}
\usepackage{stmaryrd}
\usepackage{colonequals}
\usepackage{mathtools}
\usepackage{siunitx}
\usepackage{booktabs}
\usepackage{verbatim}
\usepackage{varioref}
\usepackage{tikz}
\usetikzlibrary{cd,babel}
\usetikzlibrary{matrix,calc,arrows}

\DeclareMathAlphabet{\mathbx}{U}{bbold}{m}{n}

\usepackage{leading}
\leading{14.5pt}

\makeatletter
\AtBeginDocument{%
  \expandafter\let\csname[\endcsname\relax
  \expandafter\let\csname]\endcsname\relax
  \expandafter\DeclareRobustCommand\csname[\expandafter\endcsname\expandafter{%
    \csname begin\endcsname{equation}%
  }%
  \expandafter\DeclareRobustCommand\csname]\expandafter\endcsname\expandafter{%
    \csname end\endcsname{equation}%
  }%
}
\makeatother
\mathtoolsset{showonlyrefs,showmanualtags}

\usepackage{hyperref}
\hypersetup{
  colorlinks,
  linkcolor=-red!75!green!50,
  citecolor=-red!75!green!50,
  urlcolor=green!30!black
}

\usepackage[nobysame]{amsrefs}

\newtheorem{Theorem}{Theorem}[section]
\newtheorem{Definition}[Theorem]{Definition}
\newtheorem{Proposition}[Theorem]{Proposition}
\newtheorem{Lemma}[Theorem]{Lemma}
\newtheorem{Corollary}[Theorem]{Corollary}
\newtheorem{TheoremIntro}{Theorem}

\theoremstyle{remark}
\newtheorem{Example}[Theorem]{Example}
\newtheorem{Remark}[Theorem]{Remark}

\newlist{thmlist}{enumerate}{1}
\setlist[thmlist]{
        nolistsep,
        ref={\mdseries\textup{(\emph{\roman*})}},
        label={\mdseries\textup{(\emph{\roman*})}},
        before={\advance\mathindent\leftmargin}
        }

\usepackage{showlabels}

\newcommand\ZZ{\mathbb{Z}}
\newcommand\CC{\mathbb{C}}
\renewcommand\epsilon{\varepsilon}

\newcommand\id{\mathrm{id}}

\DeclarePairedDelimiter\paren{\lparen}{\rparen}
\DeclarePairedDelimiter\abs{\lvert}{\rvert}

\DeclarePairedDelimiter\lin{\langle}{\rangle}

\DeclareMathOperator{\im}{Im}

\DeclareMathOperator{\Ext}{Ext}

\DeclareMathOperator{\Der}{Der}

\DeclareMathOperator{\Out}{OutDer}

\DeclareMathOperator{\End}{End}

\newcommand{\Diff}{\mathsf{Diff}}
\let\hom\relax
\DeclareMathOperator{\hom}{Hom}

\newcommand\lmod[1]{{}_{#1}\mathsf{Mod}}

\let\*\bullet

\newcommand\g{\mathfrak{g}}

\newcommand\HH{H\!H}

\newcommand\kk{\Bbbk}

\newcommand\X{\mathfrak{X}}

\newcommand\A{\mathcal{A}}

\newcommand\Y[1]{C_S^{#1}(L,N)}

\newcommand\xx{\hat{x}}
\newcommand\yy{\hat{y}}

\newcommand\EE{\hat{E}}
\newcommand\DD{\hat{D}}

\DeclareMathOperator{\coker}{coker}
\setitemize{nolistsep, listparindent=\parindent}

\usepackage{pifont}
\allowdisplaybreaks

\title[Lie--Rinehart and Hochschild cohomology]{Lie--Rinehart
  and Hochschild cohomology for algebras of differential operators}

\author{Francisco Kordon}
\author{Thierry Lambre}
\address{
Laboratoire de Mathématiques Blaise Pascal,
UMR6620 CNRS,
Université Clermont Auvergne,
Campus des Cézeaux,
3 place Vasarely,
63178 Aubière cedex,
France}
\email[F.\,Kordon]{francisco.kordon@uca.fr}
\email[Th.\,Lambre]{thierry.lambre@uca.fr}
\date{\today}
\begin{document}

\begin{abstract}
Let $(S,L)$ be a Lie--Rinehart algebra such that $L$ is $S$-projective and let
$U$ be its universal enveloping algebra. In this paper we present a spectral
sequence which converges to the Hochschild cohomology of $U$ with values on a
$U$-bimodule $M$ and whose second page involves the Lie--Rinehart cohomology
of the algebra and the Hochschild cohomology of $S$ with values on~$M$. After
giving a convenient description of the involved algebraic structures we use
the spectral sequence to compute explicitly the Hochschild cohomology of the
algebra of differential operators tangent to a central arrangement
of three lines.
\end{abstract}

\maketitle

\section*{Introduction}
The goal of this article is to apply homological algebra techniques for
Lie--Rinehart algebras to a problem of algebras of differential operators.  We
begin by describing a spectral sequence that converges to the Hochschild
cohomology of the enveloping algebra of a Lie--Rinehart algebra.  After that,
we focus on the algebra of differential operators $\Diff\A$ associated to a
central arrangement $\A$ of three lines. This is a graded associative algebra
that is at the same time the enveloping algebra of a Lie--Rinehart algebra: an
explicit calculation with the spectral sequence allows us to compute the
Hilbert series of its Hochschild cohomology.  We conclude by giving two other
examples of algebras in which the spectral sequence proves useful.

Let $\kk$ be a field of characteristic zero and let $\A$ be a central
hyperplane arrangement in a finite dimensional $\kk$-vector space $V$.  Let
$S$ be the algebra of coordinates on $V$ and let $Q\in S$ be a defining
polynomial for $\A$.  The arrangement $\A$ is free if the Lie algebra
$\Der\A=\{\theta\in\Der S: \theta(Q)\in QS\}$ of derivations of $S$ tangent to
$\A$ is a free $S$-module.  It is not known what makes an arrangement free,
but this condition is nevertheless satisfied in many important examples; for
instance, it is a theorem by H.\,Terao in~\cite{terao} that reflection
arrangements over $\CC$ are free. We refer to P.\,Orlik and H.\,Terao's
book~\cite{OT} for a general reference of hyperplane arrangements.

The algebra $\Diff\A$ of differential operators tangent to an
arrangement~$\A$, first considered by F.\,J.\,Calderón-Moreno in
\cite{calderon}, is the algebra of differential operators on $S$ which
preserve the ideal $QS$ of $S$ and all its powers. We are interested in the
Hochschild cohomology of $\Diff\A$ when $\A$ is free.

The first and simplest example of a free arrangement is that of a central line
arrangement, that is, when $V=\kk^2$. Let $l$ be the number of lines of such
an arrangement: for $l\geq5$, the Hochschild cohomology of $\Diff\A$ has been
obtained as a Gerstenhaber algebra by the first author and
M.\,Su\'arez-Álvarez in~\cite{koma} starting from a projective resolution of
$\Diff\A$ as a bimodule over itself by means of explicit calculations that
exploit a graded algebra structure on $\Diff\A$, but the calculations
performed in this situation seem impossible to emulate when $l=3$ or $l=4$.  
In this paper we are able to extend, in
Corollaries~\ref{coro:HHDiffA} and~\ref{coro:abelian}, some of these results to the most
complicated case, which is~when~$l=3$:

\begin{TheoremIntro}\label{intro:diffA}
Let $\A$ be a central arrangement of three lines. The Hilbert series
of $\HH^\*(\Diff\A)$ is 
\(
h_{\HH^\*(\Diff\A)}(t) = 1 +3t + 6t^2+4t^3.
\)
The first cohomology space $\HH^1(\Diff\A)$ is an abelian Lie algebra
of dimension three.
\end{TheoremIntro}

It is to prove Theorem~\ref{intro:diffA} that Lie--Rinehart algebras come to
into play: the pair $(S,\Der\A)$ is a Lie-Rinehart algebra. Recall that a
Lie--Rinehart algebra $(S,L)$ consists of a commutative algebra $S$ and a Lie
algebra $L$ with an $S$-module structure that acts on $S$ by derivations and
which satisfies certain compatibility conditions analogous to those satisfied
by the pair $(S,\Der S)$.  The universal enveloping algebra $U$ of a
Lie--Rinehart algebra~$(S,L)$ and the Lie--Rinehart cohomology
$H_S^\*(L,N)=\Ext_U^\*(S,N)$ are an associative algebra and a cohomology
theory that generalize the usual enveloping algebra and the Lie algebra
cohomology of the Lie algebra $L$ by taking into account its interaction
with~$S$ ---see the original paper~\cite{rinehart} by G.\,Rinehart or the more
modern exposition~\cite{hueb} by J.\,Huebschmann.

If $\A$ is free, as remarked by L.\,Narv\'aez~Macarro
in~\cite{narvaez}*{Theorem 1.3.1}, the enveloping algebra of $(S,\Der\A)$ is
isomorphic to $\Diff\A$.  To compute the Hochschild cohomology
in~Theorem~\ref{intro:diffA} above we employ a strategy that gives rise to a
general method to approach this kind of computations: we construct, in
Corollary~\ref{coro:spectral}, a spectral sequence converging~to~the
Hochschild cohomology~$H^\*(U,M)$ of the enveloping algebra~$U$ with values on
an $U$-bimodule~$M$.  For this sequence we need an $U$-module structure on
$H^\*(S,M)$, the Hochschild cohomology of $S$ with values on $M$. This
$U$-module structure is constructed using an injective resolution of $M$ by
$U$-bimodules and we see in Theorem~\ref{thm:comparison} that it can be
computed explicitly from a projective resolution of $S$ by $S$-bimodules.
Moreover, the action of each $\alpha\in L$ on $H^\*(S,M)$,
computed using projectives, by the endomorphism~$\nabla_\alpha^\*$ given in
Remark~\ref{def:alphanabla} turns out suitable for computations.

\begin{TheoremIntro}\label{intro:spectral}
Let $(S,L)$ be a Lie--Rinehart pair such that $L$ is an $S$-projective module
and let $M$ be an $U$-bimodule.
There exist a $U$-module structure on $H^\*(S,M)$ and 
a first-quadrant spectral sequence $E_\*$ converging to
$H^\bullet(U,M)$ with second page
\[
  E_2^{p,q}
	= H_S^p(L,H^q (S,M)).
\]
\end{TheoremIntro}

We give two other applications of
Theorem~\ref{intro:spectral}.  First, in
Subsection~\ref{sec:Ah} we compute the Hochschild cohomology of a family of
subalgebras of the Weyl algebra over a field of characteristic zero, that is,
the algebras~$A_h$ generated by elements $x$ and $y$ satisfying the relation
$yx-xy=h$ for a given $h\in\kk[x]$.  These algebras have been studied by G.\,
Benkart, S.\, Lopes and M.\,Ondrus in the series of articles that start with
\cite{ah1} for a field of arbitrary characteristic and, more recently,
S.\,Lopes and A.\,Solotar in~\cite{lopes-solotar} have described their
Hochschild cohomology, with special emphasis on the Lie module structure of
the second cohomology space over the first one, also in arbitrary
characteristic. Some of the expressions we provide were nevertheless not found
before and might be of interest.  Second, in Subsection~\ref{sec:U} we recover
in a more direct and clear way a result by the second author and P.\,Le~Meur
in~\cite{th-pa} that states that the enveloping algebra~$U$ of a Lie--Rinehart
algebra~$(S,L)$ has Van den Bergh duality in dimension $n+d$ if $S$ has Van
den Bergh duality in dimension $n$ and $L$ is finitely generated and
projective with constant rank~$d$.

\bigskip

Let us outline the organization of this article.  In
Section~\ref{sec:l-r} we recall the definition of Lie--Rinehart pairs, their
universal enveloping algebras and their cohomology theory. In
Sections~\ref{sec:lie} and~\ref{sec:spectral} we describe the module structure
on $H^\*(S,M)$ and present the spectral sequence. After proving some useful
lemmas regarding eulerian modules in Section~\ref{sec:eulerian} we devote
Section~\ref{sec:ha} to the computation of the Hochschild cohomology of the
algebra of differential operators of a central arrangement of three lines.
Finally, in Section~\ref{sec:applications} we provide the two other
applications described
above.

\bigskip

We will denote the tensor product over the base field $\kk$ simply by
$\otimes$ or, sometimes, by~$|$.  Unless it is otherwise specified, all vector
spaces and algebras will be over $\kk$.  Given an associative algebra $A$, the
enveloping algebra $A^e$ is the vector space $A\otimes A$ endowed with the
product~$\cdot$ defined by $a_1\otimes a_2\cdot b_1\otimes b_2 = a_1b_1\otimes
b_2a_2$, so that the category of $A^e$-modules is equivalent to that of
$A$-bimodules. The Hochschild cohomology of $A$ with values on an $A^e$-module
$M$ is defined as $\Ext_{A^e}^\*(A,M)$ and will be denoted by $H^\*(A,M)$ or,
if $M=A$, by $\HH^\*(A)$. The book~\cite{weibel} by C. Weibel may serve as
general reference on this subject.

\bigskip

The first author heartfully thanks his PhD advisor M\@.~Suárez-Álvarez for
his collaboration, fruitful suggestions and overall help.  We thank the
Universit\'e Clermont Auvergne for hosting the first author in a postdoctoral
position at the Laboratoire de Math\'ematiques Blaise Pascal during the year
2019-2020. Part of this work was done during the time the first author was
supported by a full doctoral grant by CONICET and by the projects PIP-CONICET
12-20150100483, PICT 2015-0366 and UBACyT~20020170100613BA.

\section{Lie--Rinehart algebras}\label{sec:l-r}

We begin by recalling some basic facts about Lie-Rinehart algebras available
in~\cite{rinehart} and in~\cite{hueb}. Until Section~\ref{sec:spectral} we
assume $\kk$ to be a field of arbitrary characteristic.

\begin{Definition}
Let $S$ and $(L,[-,-])$  be a commutative and a Lie algebra endowed with
a morphism of Lie algebras $L\to\Der_\kk(S)$ that
we write $\alpha\mapsto\alpha_S$    
and  
a left $S$-module structure on $L$
which we simply denote by juxtaposition.
The pair $(S,L)$ is a \emph{Lie--Rinehart algebra} if the equalities
\begin{align}
  &(s\alpha )_S(t) = s\alpha_S(t),
  &&[\alpha,s\beta] = s[\alpha,\beta] + \alpha_S(s)\beta
\end{align} 
hold whenever $s,t\in S$ and $\alpha, \beta \in L$.
\end{Definition}

\begin{Definition}
  Let $(S,L)$ be a Lie-Rinehart algebra.
  A \emph{Lie--Rinehart module} ---or $(S,L)$-module--- is a
  vector space $M$ that is at the same time an $S$-module and an $L$-Lie module
  in such a way that
  \begin{align}\label{eq:L-Rmodule}
    &\left( s\alpha \right)\cdot m = s\cdot(\alpha\cdot m), 
    &&\alpha\cdot (s\cdot m) = (s\alpha)\cdot m + \alpha_S(s)\cdot m
  \end{align}  
  for $s\in S$, $\alpha\in L$ and $m\in M$. 
\end{Definition}

\begin{Theorem}
Let $(S,L)$ be a Lie-Rinehart algebra.
\begin{thmlist}
\item There exists an associative algebra $U=U(S,L)$, the \emph{universal
  enveloping algebra of $(S,L)$}, endowed with a
  morphism of algebras $i:S\to U$ and a morphism of Lie algebras $j:L\to U$
  that satisfies, for $s\in S$ and $\alpha\in L$,
  \begin{align}\label{eq:universal}
    &i(s)j(\alpha) 	= j(s\alpha),
    && j(\alpha)i(s) -i(s)j(\alpha) = i(\alpha_S(s)) 
  \end{align}
  and universal with these properties. 

\item The category of $U$-modules is isomorphic to that of $(S,L)$-modules.
\end{thmlist}
\end{Theorem}

\begin{Example}\label{ex:weyl}
The obvious actions of $S$ and $L$ make of $S$ an $U$-module. If $\g$ is a
Lie algebra then $(\kk,\g)$ is a Lie--Rinehart algebra whose enveloping
algebra is simply the usual enveloping algebra of $\g$.  If
$S=\kk[x_1,\dots,x_n]$ then the full Lie algebra of derivations $L=\Der_\kk S$
is a Lie--Rinehart algebra and its enveloping algebra is isomorphic to the
algebra of differential operators $\Diff(S)=A_n$, the $n$th Weyl~algebra. 
\end{Example}

\begin{Definition}
  Let $(S,L)$ be a Lie--Rinehart algebra with enveloping algebra $U$ and let $N$ be
  an $U$-module. The \emph{Lie--Rinehart cohomology of~$(S,L)$ with values on
  $N$} is
  \[
    H_S^\*(L,N)\coloneqq\Ext^\*_{U}(S,N).
  \]
\end{Definition}

In many important situations, some of which will be illustrated in the
examples below, $L$~is a projective $S$-module, and in this case there is a
well-known complex that computes the Lie--Rinehart cohomology.

\begin{Proposition}\label{prop:ch-ei}
  Suppose that $L$ is $S$-projective and let $\Lambda_S^\*L$ denote the exterior
  algebra of $L$ over $S$. The complex $\hom_S(\Lambda_S^\bullet L,N)$
  with Chevalley--Eilenberg differentials computes~$H_S^\*(L,N)$. 
\end{Proposition}

\begin{Example}\label{ex:manifolds}
For the Lie--Rinehart algebra $(\kk,\g)$ with $\g$ a Lie algebra, $N$ is
simply a $\g$-Lie module and the complex~$\hom_\kk(\Lambda_\kk^\bullet L,N)$
is the standard complex that computes the Lie algebra
cohomology~$H^\bullet(\g,N)$.

Given a finite dimensional manifold $M$, we obtain a Lie--Rinehart algebra
setting $S=C^\infty(M)$, the algebra of smooth functions, and
$L=\mathfrak{X}(M)$, the Lie algebra of vector fields on $M$. The enveloping
algebra of this pair is isomorphic to the algebra of globally defined
differential operators on the manifold ---see~\cite{hueb}*{\S1}. We can find
in J.\, Nestruev's~\cite{nestruev}*{Proposition 11.32} that~$L$ is finitely
generated and projective over $S$; as the complex~$\hom_S(\Lambda_S^\bullet
L,S)$ is the de Rham complex $\Omega^\*(M)$ of differential forms, the
cohomology~$H_S^\*(L,S)$ coincides with the de~Rham cohomology of~$M$.
\end{Example}

\begin{Example}\label{ex:arrangements1}
A central \emph{hyperplane arrangement} $\A$ in a finite dimensional vector
space $V$ is a finite set $\{H_1,\ldots,H_l\}$ of subspaces of codimension 1.
Let $\lambda_i:V\to\kk$ be a linear form with kernel $H_i$ for each
$i\in\{1,\ldots,l\}$. We let $S$ be the algebra of polynomial functions
on~$V$, fix a defining polynomial $Q=\lambda_1\cdots\lambda_l\in S $ for~$\A$
and consider the Lie~algebra 
\[
  \Der \A \coloneqq 
  \{ \theta\in\Der_\kk(S)~: 	\text{$Q$ divides $\theta (Q)$}	\}
\]
of derivations tangent to the arrangement. The pair
$(S,\Der\A)$ is a Lie--Rinehart algebra, as one can readily check.

An arrangement $\A$ is \emph{free}, by definition, if $\Der\A$ is a free
$S$-module. In that case, as in~\cite{narvaez}*{Theorem 1.3.1}, the enveloping
algebra of $(S,\Der\A)$ is isomorphic to the \emph{algebra of differential
operators tangent to the arrangement} $\Diff\A$, that is, the algebra of
differential operators on $S$ which preserve the ideal $QS$ of $S$ and all its
powers.  As seen in \cite{calderon} or by M.\,Suárez-Álvarez
in~\cite{differential-arrangements}, it coincides with the associative algebra
generated inside the algebra $\End_\kk(S)$ of linear endomorphisms of the
vector space $S$ by $\Der\A$ and the set of maps given by left multiplication
by elements of $S$.  

For the Lie--Rinehart algebra $(S,L)$ associated to a free hyperplane
arrangement $\A$, the complex $\hom_S(\Lambda_S^\bullet L,S)$ is the complex
of logarithmic forms $\Omega^\bullet(\A)$, and its cohomology is isomorphic to
the Orlik--Solomon algebra of $\A$ ---here we refer to J.\,Wiens and
S.\,Yuzvinsky's~\cite{wy}.  When $\kk=\mathbb{C}$, this algebra is, in turn,
isomorphic to the cohomology of the complement of the arrangement, as proved
by P.\,Orlik and L.\,Solomon~in~\cite{OS}.
\end{Example}

\section{The \texorpdfstring{$U$-module structure on
$H^\bullet(S,M)$}{U-module structure on H(S,M)}}\label{sec:lie}

Let~$(S,L)$ be a Lie--Rinehart algebra such that~$L$ is a projective
$S$-module. Let~$U$ be
its enveloping algebra and~$M$ be an~$U^e$-module. Since the inclusion
of~$S$ in~$U$ is a morphism of algebras we can regard~$M$ as an~$S^e$-module
and consider the Hochschild cohomology of~$S$ with values
on~$M$, denoted as before by~$H^\bullet(S,M)$. 
In this section we first construct an $U$-module structure on $H^\bullet(S,M)$
from an
$U^e$-injective resolution of $M$;
afterwards, we construct $S$- and~$L$-module structures on~$H^\*(S,M)$ from
an~$S^e$-projective resolution of~$S$; finally, we show that these
induce an $U$-module
structure that coincides with the one we have using injectives: this will
allow us to compute the latter in~practice.

\subsection{Using \texorpdfstring{$U^e$}{U\^{}e}-injective modules}
\label{subsec:lie-inj}
The second author and P.\,Le~Meur introduce in \cite{th-pa}*{Lemma 3.2.1} the
functor
\[\label{eq:Gfunctor}
  \begin{aligned}
    G = \hom_{S^e}(S,-) : \lmod{U^e} &\to \lmod{U},
  \end{aligned}
\]
where for an $U^e$-module $M$ the left $L$-Lie module and
left $S$-module structures on $\hom_{S^e}(S,M)$ are defined by the rules
\begin{equation}\label{eq:U-structure}
  \begin{aligned}
    &(\alpha\cdot\varphi) (s)
    = (\alpha\otimes 1)\cdot \varphi(s) 
    -(1\otimes\alpha)\cdot\varphi(s) 
    - \varphi\left( \alpha_S(s) \right),\\ 
    &(t\cdot \varphi)(s) = (t\otimes 1)\cdot\varphi(s)
  \end{aligned}
\end{equation}
for $\alpha\in L$, $\varphi\in\hom_{S^e}(S,M)$ and $s,t\in S$.

\begin{Proposition}\label{prop:injectives}
  Let $M$ be an $U^e$-module and let $M\to I^\bullet$ be an injective resolution
  of $M$ as an $U^e$-module.  The cohomology of the complex
  $G(I^\*)=\hom_{S^e}(S,I^\*)$ is the Hochschild cohomology~$H^\*(S,M)$.
\end{Proposition}

\begin{proof}
  Let $I$ be an injective $U^e$-module.  The functor $\hom_{S^e}(-,I)$ is
  naturally isomorphic to $\hom_{U^e}(U^e\otimes_{S^e}-,I)$, which is the
  composition of the exact functor $\hom_{U^e}(-,I)$ and $U^e\otimes_{S^e}-$.
  Now, the PBW-theorem in \cite{rinehart}*{\S3} ensures that $U$ is a projective
  $S$-module and, using Proposition IX.2.3 of H.\,Cartan and
  S.\,Eilenberg's~\cite{cartan-eilenberg} , we obtain that $U^e$ is
  $S^e$-projective. As a consequence of this, the functor $U^e\otimes_{S^e}-$ is
  exact and therefore $\hom_{S^e}(-,I)$ is exact as well. This implies that $M\to
  I^\bullet$ is in fact a resolution of $M$ by $S^e$-injective modules, so that
  $H^\*(\hom_{S^e}(S,I^\*)) =\Ext_{S^e}(S,M)$.
\end{proof}

From Proposition~\ref{prop:injectives} and the functoriality of
$G=\hom_{S^e}(S,-)$ we can conclude that if $M\to I ^\*$ is an
$U^e$-injective resolution then the $U$-module structure on
$\hom_{S^e}(S,I^\bullet)$ defined in \eqref{eq:U-structure} induces an
$U$-module structure~on~$H^\*(S,M)$:

\begin{Corollary}\label{coro:lie-inj}
Let $M$ be an $U^e$-module and let $M\to I^\bullet$ be an $U^e$-injective
resolution. Let $j\geq0$, $u\in U$ and denote the class in $H^j(S,M)$ of
$\varphi\in \hom_{S^e}(S,I^j)$ by $\bar\varphi$.
Defining
\(
u\cdot\bar\varphi
\)
to be the class of $u\cdot\varphi$ as defined in
\eqref{eq:U-structure} we obtain an $U$-module structure on $H^j(S,M)$.
\end{Corollary}

\subsection{Using \texorpdfstring{$S^e$}{S\^{}}-projective modules}
\label{subsec:lie-proj}
In this subsection we define $S$- and $L$-module structures on $H^\*(S,M)$
using projectives. To see that these structures are compatible as
in~\eqref{eq:L-Rmodule} we will show that the are equal to the ones in
Subsection~\ref{subsec:lie-inj} using injectives and conclude that they
determine an $U$-module structure.  

\subsubsection{The $S$-module structure}
We start by letting $P_\*\to S $ be an $S^e$-projective resolution. For each
$i\geq0$ there is a left $S$-module structure on $\hom_{S^e}(P_i,S)$ given by
\[\label{eq:Sprojectives}
  (s\cdot\phi) (p) = s\phi(p)
  \qquad\text{for $s\in S$, $\phi\in\hom_{S^e}(P_i,S)$ and $p\in P_i$.}
\]
With this structure the differentials in the complex $\hom_{S^e}(P_\*,S)$
become $S$-linear and therefore the cohomology of this complex, which is
canonically isomorphic to $H^\*(S,M)$, inherits an $S$-module structure. It is
straightforward to verify that this structure does not depend on the choice of
the projective resolution.

\subsubsection{\texorpdfstring{$\delta$}{Delta}-liftings}\label{sec:deltaops}
To give an $L$-Lie module structure on $H^\*(S,M)$ using projectives we will
use the tools developed by M.\,Suárez-Álvarez in~\cite{mariano-extra}. 
Let $A$ be an algebra and $\delta:A\to A$ a derivation. Given an $A$-module
$V$, we say that a linear map $f:V\to V$ is a
\emph{$\delta$-operator} if for every $a\in A$ and $v\in V $ we have
\[
  f(av) = \delta(a) v + a f(v).
\]
If, moreover, $\varepsilon: P_\bullet\to V$ is an $A$-projective resolution of
$V$, a \emph{$\delta$-lifting} of $f$ to $P_\bullet$ is a family of
$\delta$-operators $f_\bullet =(f_i:P_i\to P_i, i\geq 0)$ such that the
following diagram commutes:
\[
  \begin{tikzcd}[column sep=1.75em] 
    \cdots \arrow[r] 
    & P_1\arrow[r]\arrow[d,"f_1"]
    & P_0\arrow[r]\arrow[d,"f_0"]
    & V\arrow[d,"f"]	\\
    \cdots \arrow[r] 
    & P_1\arrow[r]
    & P_0\arrow[r]
    & V	
  \end{tikzcd}
\]

The construction in~\cite{mariano-extra}*{\S1} proceeds then 
as follows. Given an algebra $A$ with a
derivation $\delta$, a $\delta$-operator $f:V\to V$ and a projective
resolution $P_\*\to V$, a $\delta$-lifting $f_\*$ of $f$ to $P_\*$
is shown to always exist. This $\delta$-lifting
gives rise to an endomorphism $f_\*^\sharp$ of the
complex $\hom_A(P_\*,V)$ defined for $i\geq0$ and $\varphi\in\hom_A(P_i,V)$
by $f_i^\sharp(\varphi)=f\circ\varphi-\varphi\circ f_i$. Moreover,
$f_\*^\sharp$ induces an endomorphism $\nabla_f^\*$ of the cohomology
$\Ext_A^\*(V,V)$ which, conveniently, does not depend neither on the choice of
the $\delta$-lifting or the projective resolution.

We will now generalize this construction so that we can adapt it to our needs.
Let us first recall two simple but fundamental results in the following
Lemma.

\begin{Lemma}[\cite{mariano-extra}*{\S1.4,\S1.6}]\label{prop:deltaops}
  Let $V$ be a left $A$-module, let $f:V\to V$ be a $\delta$-operator and let
  $\varepsilon:P_\bullet\to V$ be a projective resolution. 
  \begin{thmlist}
  \item 
    There exists a $\delta$-lifting $f_\*$ of $f$ to $P_\*$.


  \item If $\varepsilon':P'_\*\to V$ is another projective resolution, $f_\*$
    and $f'_\*$ are $\delta$-liftings of $f$ to $\varepsilon$ and $\varepsilon'$
    and $h_\*:P'_\*\to P_\*$ is an $A$-linear lifting of $\id_V:V\to V$ then
    \(
    f_\*h_\*-h_\*f'_\* : P'_\*\to P_\*
    \)
    is an $A$-linear lifting of the zero map $0:V\to V$.
  \end{thmlist}
\end{Lemma}

\begin{Proposition}\label{prop:fgsharp}
  Let $V$ and $W$ be two $A$-modules, $f:V\to V$ and $g:W\to W$ two
  $\delta$-operators and $P_\*\to V $ an $A$-projective resolution.  Let
  $f_\*=(f_i)_{i\geq0}$ be a $\delta$-lifting of $f$ to $P_\*$ provided by
  Proposition~\ref{prop:deltaops}.
  \begin{thmlist}
  \item
    There is an endomorphism $(f_\*,g) = \left( (f_i,g)
    \right)_{i\geq0}$ of the complex of vector spaces $\hom_A(P_\*,W)$ such that
    if $i\geq0$ and $\phi\in\hom_A(P_i,W)$ then 
    \[\label{eq:fgsharp}
      (f_i,g) (\phi)= g\circ\phi-\phi\circ f_i.  
    \]

  \item
    The map~$\nabla_{(f,g)}^\*:\Ext_A^\*(V,W)\to \Ext_A^\*(V,W)$ 
    induced by $(f_\*,g)$ in cohomology is independent of
    the choice of the projective resolution~$P_\*\to S$ and the
    $\delta$-lifting~$f_\*$.
  \end{thmlist}
\end{Proposition}

\begin{proof}
  Let $i\geq0$. As both $g:W\to W$ and $f_i:P_i\to P_i$ are $\delta$-operators
  and $\phi$ is $A$-linear, the difference $(f_i,g) (\phi)=
  g\circ\phi-\phi\circ f_i$ is $A$-linear. That $(f_\*,g) $ is a morphism
  of complexes is an immediate consequence of the fact that so is $f_\*$.

  For the second assertion we let $\varepsilon':P'_\*\to V$ be another
  $A$-projective resolution of~$V$, $f'_\*$ be another $\delta$-lifting of $f$
  to $P_\*$ and $(f_\*',g)$ be the graded endomorphism of
  $\hom_{A}(P'_\*,W)$ in \eqref{eq:fgsharp}. 
  We claim that if $h : P'_\*\to P_\*$ is a morphism of complexes lifting the
  identity of $S$ then the diagram
  \[\label{eq:diag:homot}
    \begin{tikzcd}[column sep=4em] 
      \hom_{S^e}(P_\*,W)\arrow[r,"{(f_\*,g)}"]\arrow[d,"h_\*^*"]
      &\hom_{S^e}(P_\*,W)\arrow[d,"h_\*^*"]	\\
      \hom_{S^e}(P'_\*,W)\arrow[r,"{(f_\*',g)}"]
      &\hom_{S^e}(P'_\*,W)	
    \end{tikzcd}
  \]
  commutes up to homotopy.  

  Proposition~\ref{prop:deltaops} tells us
  that $z_\* \coloneqq f_\*h_\* -h_\*f'_\*:P'_\*\to P_\*$ is an $A$-linear
  lifting of $0:V\to V$ and therefore
  \(
  z_\*^*:\hom_{A}(P_\*,W)\to\hom_{A}(P'_\*,W)
  \)
  is homotopic to zero. To prove the claim it is then enough to show that 
  \[\label{eq:homot}
    (f_i',g)\circ h_i^*-h_i^*\circ(f_i,g)= z^*_i
    \quad\text{for each $i\geq0$,}
  \]
  so that the zero-homotopic map $z_\*^*$ is the failure in the commutativity of
  the diagram~\eqref{eq:diag:homot}.  We have, for $\phi\in\hom_{S^e}(P_i,W)$,
  \begin{align*}
    \MoveEqLeft[4]
    \left( (f_i',g)\circ h_i^*-h_i^*\circ(f_i,g)\right) (\phi)
    = (f_i',g)(\phi \circ h_i) - h_i^*((f_i,g)(\phi)) \\
    &= g \circ \left( \phi\circ h_i  \right) 
    - \left( \phi\circ h_i  \right)\circ f'_i 
    - \left( g \circ \phi \right)\circ h_i
    + \left( \phi\circ f_i \right)\circ h_i \\
    &= \phi \circ f_i\circ h_i - \phi\circ h_i\circ f_i' \\
    &= (h_i^*f_i^* - f_i'^* h_i^*) (\phi)
    =z^*_i(\phi). 
  \end{align*}
  This proves the claim, and it follows at once that 
  the endomorphisms that $(f_\*,g)$ and $(f_\*',g)$
  induce on~$\Ext_A^\*(V,W)$ are equal.
\end{proof}


\subsubsection{The $L$-Lie module structure}
Let $(S,L)$ be a Lie--Rinehart algebra, $M$ be an $U^e$-module and
$\alpha\in L$. To adapt the construction of
Subsection~\ref{sec:deltaops} to our situation we 
recall that $\alpha$ acts on $S$ by the derivation
$\alpha_S:S\to S$ and consider the following assertions.

\begin{enumerate}[label=(\roman*),font=\small\itshape]
\item
  The map 
  \(\label{eq:alphaext}
  \alpha_S^e = \alpha_S\otimes 1 + 1\otimes \alpha_S : S^e\to S^e
  \)
  is a derivation.
\item 
  Viewing $S$ as an $S^e$-module via $(s_1\otimes s_2)\cdot t\coloneqq s_1ts_2$,
  the derivation~$\alpha_S:S\to S$ becomes an $\alpha_S^e$-operator.

\item
  The map $\alpha_M:M\to M$ such that
  $\alpha_M(m)=(\alpha\otimes 1)\cdot m -(1\otimes\alpha )\cdot m$
  satisfies 
  \[\label{eq:corchete}
    \alpha_M\left( (s\otimes t)\cdot m  \right)
    =\alpha_S^e(s\otimes t)\cdot m + (s\otimes t) \cdot \alpha_M(m)
    \qquad\text{for $s,t\in S$ and
    $m\in M$}, 
  \]
  which is to say that, regarding $M$ as an $S^e$-module, $\alpha_M$ is an
  $\alpha_S^e$-operator.
\end{enumerate}

The first two claims can be proved with a straightforward calculation; for
the third one, we let $\alpha$, $s$, $t$ and $m$ as before and see that
\begin{align}
  \MoveEqLeft \alpha_M\left( (s\otimes t)\cdot m  \right)
  =\left( (\alpha\otimes 1-1\otimes \alpha) (s\otimes t)\right)\cdot m
  \\
  &= (\alpha s\otimes t - s\otimes t\alpha)\cdot m
  = \left((\alpha (s)+s\alpha) \otimes t - s\otimes (\alpha t -\alpha(t))\right)\cdot m
  \\
  &=\alpha^e(s\otimes t)\cdot m +(s\alpha\otimes t-s\otimes \alpha t)\cdot m
  =\alpha^e(s\otimes t)\cdot m + s\otimes t \cdot \alpha_M(m)
\end{align}
since $\alpha_S(s)=s\alpha-\alpha s$, as in~\eqref{eq:universal}.

We may now specialize Proposition~\ref{prop:fgsharp} to our situation. We take
\begin{align}
  &A=S^e,
  && \delta = \alpha_S^e :S^e\to S^e ,
  &&V = S, \\
  & f=\alpha_S :S\to S,
  && W = M,
  && g = \alpha_M:M\to M
\end{align}
and from this we obtain the maps $\alpha_\*^\sharp
\coloneqq(f_\*^\sharp,g)$ and $\nabla_\alpha^\* \coloneqq\nabla_{(f,g)}^\*$.
More concretely:

\begin{Remark}\label{def:alphanabla}
Let $\alpha\in L$, $M$ an $U^e$-module and $\varepsilon:P_\*\to S$ an
$S^e$-projective resolution. Let $\alpha_\*$ be an $\alpha^e_S$-lifting of
$\alpha_S:S\to S $ to~$P_\*$, that is, a morphism of complexes
$\alpha_\*=(\alpha_{q}:P_q\to P_q)_{q\geq0}$ such that
$\varepsilon\circ\alpha_0= \alpha_S\circ \varepsilon$ and for each $q\geq0$,
$s$, $t\in S$ and $p\in P_q$
\[
  \alpha_q( ( s\otimes t)\cdot p)
  =\left( \alpha_S(s)\otimes t + s\otimes\alpha_S(t) \right)\cdot p 
  +(s\otimes t)\cdot p.
\]
Denote by
$\alpha\otimes 1-1\otimes\alpha:M\to M$ the map such that $m\mapsto
(\alpha\otimes 1-1\otimes \alpha)\cdot m$.
The endomorphism $\alpha_\*^\sharp$ of $\hom_{S^e}(P_\*,M)$
is given for each $q\geq0$ by
\begin{align}\label{eq:alphasharp}
  \alpha^\sharp_q (\phi)
  = (\alpha \otimes 1 - 1\otimes\alpha)\circ\phi
  - \phi\circ\alpha_q,
\end{align}
and the map~$\nabla_\alpha^\*:H^\*(S,M)\to H^\*(S,M)$ is the unique graded
endomorphism such that 
\[\label{eq:alphanabla}
  \nabla_\alpha^q([\phi]) = [\alpha_q^\sharp(\phi)],
\]
where [-] denotes class in cohomology.
\end{Remark}

\begin{Proposition}\label{prop:lieprojec}
Let $\End\left(H^\*(S,M)  \right) $ be the Lie algebra of linear
endomorphisms of $H^\*(S,M)$ with Lie structure given by the commutator.
The map
\(
  \nabla:
  L \to \End \left(H^\*(S,M)  \right) 
\)
defined by $\alpha\mapsto \nabla^\*_\alpha$
is a morphism of Lie algebras.
\end{Proposition}

\begin{proof}
Let $\alpha, \beta\in L$ and call $\gamma=[\alpha,\beta]$. Let $\alpha_\*$,
$\beta_\*$ and $\gamma_\*$ be $\alpha^e$, $\beta^e$ and $\gamma^e_S$-liftings,
respectively. Observe that $\gamma_\*$ is not necessarily 
the commutator of $\alpha_\*$ and $\beta_\*$.  Let $\alpha_\*^\sharp$,
$\beta_\*^\sharp$ and $\gamma_\*^\sharp$ be the endomorphisms of
$\hom_{S^e}(P_\*,M)$ defined as in~\eqref{eq:alphasharp} and
consider the endomorphism $\theta^\*$ of $\hom_{S^e}(P_\*,M)$ such that
if $i\geq 0$ and $\phi\in\hom_{S^e}(P_i,M)$
\[
  \theta^i (\phi) 
  = (\gamma\otimes1-1\otimes\gamma)\circ \phi
  -\phi\circ\left( \alpha_i\circ\beta_i - \beta_i\circ\alpha_i \right).
\]
A straightforward calculation shows that the commutator
$\alpha_\*\circ\beta_\*-\beta_\*\circ\alpha_\*$ is a $\gamma^e_S$-lifting of
$\gamma$ and therefore Proposition~\ref{prop:fgsharp} tells us that
$\theta^\*$ and $\gamma_\*^\sharp$ induce the same endomorphism on cohomology.
We claim that in fact
$\theta^\* = \alpha_\*^\sharp\circ\beta_\*^\sharp-
\beta_\*^\sharp\circ\alpha_\*^\sharp$. Indeed, for $i\geq0$ and
$\phi\in\hom_{S^e}(P_i,M)$ 
\begin{align*}
  \alpha_i^\sharp(\beta_i^\sharp(\phi))
  &= (\alpha\otimes1-1\otimes\alpha)\circ\beta_i^\sharp(\phi)
  -\beta_i^\sharp(\phi)\circ\alpha_i  \\
  &\!\begin{multlined}[.8\displaywidth]
    =(\alpha\otimes1-1\otimes\alpha)\circ
    \left((\beta\otimes1-1\otimes\beta)\circ\phi- \phi\circ\beta_i \right)
    \\
    -(\beta \otimes 1 - 1\otimes\beta)\circ\phi\circ\alpha_i
    - \phi\circ\beta_i\circ\alpha_i  
  \end{multlined}\\
  &\!\begin{multlined}[.8\displaywidth]
    =(\alpha\beta\otimes 1 - \alpha\otimes\beta 
    -\beta\otimes\alpha + \alpha\otimes\beta) \circ\phi
    -(\alpha \otimes 1 - 1\otimes\alpha)\circ\phi\circ\alpha_i\\
    -(\beta \otimes 1 - 1\otimes\beta)\circ\phi\circ\alpha_i
    -\phi\circ\beta_i\circ\alpha_i
  \end{multlined}
\end{align*}
These two expressions together with the equality
$\alpha\beta-\beta\alpha=\gamma$ in $U$ allow us to conclude that
$\alpha_i^\sharp(\beta_i^\sharp(\phi)) - \beta_i^\sharp(\alpha_i^\sharp(\phi))
=\theta^i(\phi)$, which proves the claim.

We conclude in this way that
\begin{align}
  \label{eq:LieH}
  H^\*(\gamma_\*^\sharp)
  &= H^\*(\theta^\*)
  = H^\*(\alpha_\*^\sharp\circ\beta_\*^\sharp
  -\beta_\*^\sharp\circ\alpha_\*^\sharp ) \\
  &= H^\*(\alpha_\*^\sharp)\circ H^\*(\beta_\*^\sharp)
  -H^\*(\beta_\*^\sharp)\circ H^\*(\alpha_\*^\sharp ),
\end{align}
in virtue of the linearity of the functor $H$. This means 
that~$\nabla_\gamma^\bullet = [\nabla_\alpha^\*,\nabla_\beta^\*]$.
\end{proof}

\begin{Example}\label{ex:alpha0}
  It is easy to describe the endomorphism $\nabla_\alpha^0$ of $H^0(S,U)$ for a
  given $\alpha\in L$.  Let us choose a resolution $P_\*$ of $S$ with $P_0=S^e$
  and augmentation $\varepsilon:S^e\to S$ defined by
  $\varepsilon(s\otimes t)=st$.  As $\alpha_S^e$ is a $\alpha_S^e$-operator and
  $\varepsilon\circ \alpha_S^e =\alpha_S\circ \varepsilon$, we may choose an
  $\alpha_S^e$-lifting with $\alpha_0=\alpha_S^e$.  According to the
  rule~\eqref{eq:alphasharp} we have
  \[\label{eq:alpha0}
    \alpha_0^\sharp (\phi)(1\otimes 1)
    =(\alpha\otimes1-1\otimes\alpha)\cdot\phi(1\otimes 1)
    \qquad\text{for all $\phi\in \hom_{S^e}(P_0,M)$.}
  \]
  Identifying, as usual, each $\phi\in\hom_{S^e}(S^e,U)$ with $\phi(1\otimes
  1)\in U$, we can view $H^0(S,U)$ as a subspace of $U$ and
  then~\eqref{eq:alpha0} tells us that 
  \(
  \nabla_\alpha^0(u)=\alpha u	-u\alpha
  \)
  for all $u\in H^0(S,U)$.
\end{Example}

\subsection{Comparing the two actions.}
We now prove that the $S$- and $L$-module structures on $H^\*(S,M)$
constructed in~Subsection~\ref{subsec:lie-proj} using projectives are equal to
those induced by the $U$-module structure in~Subsection~\ref{subsec:lie-inj}
using injectives. As a consequence, this shows that the actions of $S$ and $L$
using projectives satisfy compatibility relations~\eqref{eq:L-Rmodule}.

\begin{Theorem}\label{thm:comparison}
Suppose $L$ is $S$-projective. The $S$- and $L$-module structures on
$H^\*(S,M)$ determined by~\eqref{eq:U-structure} using $U^e$-injective modules
are equal to those given in~\eqref{eq:Sprojectives} and~\eqref{eq:alphanabla}
using $S^e$-projective modules.
\end{Theorem}

\begin{proof}
We will only prove that the $L$-module structures coincide ---that the
$S$-module structures are equal too is analogous and simpler.  To begin with,
we fix an $U^e$-injective resolution $\eta:M\to I^\*$, an $S^e$-projective
resolution $\varepsilon:P_\*\to S$ and $\alpha\in L$.
In~\eqref{eq:alphasharp}, we give endomorphisms of complexes
$\alpha_\*^\sharp$ of $\hom_{S^e}(P_\bullet,M)$ and of
$\hom_{S^e}(P_\bullet,I^j)$ for each $j\geq 0$ ---we denote them the same
way--- which induce the map $\nabla_\alpha^\*$ on their cohomologies
$H^\*(S,M)$ and $H^\*(S,I^j)$.  We first claim that the map 
\[
  \eta_* :
  \hom_{S^e}(P_\bullet,M)\ni\phi 
  \longmapsto \eta\circ\phi\in \hom_{S^e}(P_\bullet,I^\*)
\]
satisfies, for each $i\geq 0$ and $\phi\in\hom_{S^e}(P_i,M)$,
\[\label{eq:etaequi}
  \eta_*(\alpha_i^\sharp(\phi)) 
  = \alpha_i^\sharp(\eta_*(\phi)).
\]
Indeed, since $\eta$ is a morphism of $U^e$-modules it commutes with 
$1\otimes\alpha-\alpha\otimes 1$ and thus
\begin{align*}
  \eta_*(\alpha_i^\sharp(\phi)) 
  &=\eta\circ (\alpha \otimes 1 - 1\otimes\alpha)\circ\phi
  - \eta\circ\phi\circ\alpha_i \\
  &=(\alpha \otimes 1 - 1\otimes\alpha)\circ\eta\circ \phi
  - \eta\circ\phi\circ\alpha_i 
  = \alpha_i^\sharp(\eta_*(\phi)).
\end{align*}
Let us see that, on the other hand, the map
\[
  \varepsilon^* :
  \hom_{S^e}(S,I^\*)\ni\varphi 
  \longmapsto \varphi\circ \varepsilon \in \hom_{S^e}(P_\bullet,I^\*)
\]
satisfies that for each $\varphi\in \hom_{S^e}(S,I^\*)$
\[\label{eq:epsiequi}
  \varepsilon^*(\alpha\cdot\varphi) 
  = \alpha_0^\sharp(\varepsilon^*(\varphi)).
\]
Since $\alpha_\*$ is a lifting of $\alpha_S:S\to S$ to $P_\*$, we have
that $\alpha\circ\varepsilon = \varepsilon\circ\alpha_0$ and
\begin{align*}
  \varepsilon^*(\alpha\cdot\varphi) 
  &= (\alpha \otimes 1 - 1\otimes\alpha)\circ\varphi\circ\varepsilon
  - \varphi\circ\alpha\circ\varepsilon \\
  &=(\alpha \otimes 1 - 1\otimes\alpha)\circ\varphi\circ\varepsilon
  - \varphi\circ\varepsilon\circ\alpha_0
  =\alpha_0^\sharp(\varepsilon^*(\varphi)).
\end{align*}
As the morphisms of complexes $\varepsilon^*$ and $\eta_*$ are
quasi-isomorphisms, the fact that they are equivariant with respect to the
actions of $\alpha$ ---as shown by~\eqref{eq:etaequi} and~\eqref{eq:epsiequi}---
allows us to conclude that the two actions of $L$ on $H^\*(S,M)$ coincide.
\end{proof}

\section{The spectral sequence}\label{sec:spectral}

Let $(S,L)$ be a Lie--Rinehart algebra, let $U$ be its enveloping algebra and
let $M$ be an $U^e$-module. In this section we construct a spectral sequence
which converges to the Hochschild cohomology of $U$ with values on $M$ and
whose second page involves the Lie--Rinehart cohomology of $(S,L)$ and the
Hochschild cohomology of $S$ with values~on~$M$.  

Recall that in~\eqref{eq:Gfunctor} we considered a functor
$G:\lmod{U^e}\to\lmod{U}$ defined on objects as $G(M)=\hom_{S^e}(S,M)$.
We now consider the functor
\[\label{eq:Ffunctor}
  \begin{aligned}
    F:\lmod{U} &\to\lmod{U^e} \\
    F(N) &= U \otimes_S N
  \end{aligned}
\]
where we give to $U\otimes_SN$ the $U^e$-module structure
in~\cite{hueb:duality}*{(2.4)}. This structure is completely determined by the
rules 
\begin{align}
  &(v\otimes 1)\cdot u\otimes_S n
  = vu\otimes_S n, \\
  &(1\otimes \alpha)\cdot u\otimes_S n
  = u\alpha\otimes_S n - u\otimes_S \alpha\cdot n,
  &&(1\otimes s)\cdot u\otimes_S n
  = u\alpha\otimes_Ss\cdot n
\end{align}
for $u,v\in U$, $n\in N$ and $\alpha\in L$.
With the
functors $G$ and $F$ at hand, we can state the very useful Proposition~3.4.1
of~\cite{th-pa}.

\begin{Proposition}\label{prop:adjunction}
  The functor $F$ is left adjoint to $G$. 
\end{Proposition}

\begin{Theorem}\label{thm:spectral}
  Assume $L$ is $S$-projective and let $N$ and $M$ be a left $U$-module and
  an $U^e$-module. There is a
  first-quadrant spectral sequence $E_\*$ converging to $\Ext_{U^e}^\*(F(N),
  M)$ with second page 
  \[
    E_2^{p,q}
    = \Ext_U^p(N,H^q(S,M)).
  \]
\end{Theorem}

\begin{proof}
Let $Q_\bullet\to N$ be an $U$-projective resolution of $N$ and let $M\to
I^\bullet$ be an $U^e$-injective resolution. Consider the double complex
\[
  X^{\bullet,\bullet} 
  = \hom_U(Q_\bullet,G(I^\bullet))
\]
and denote its total complex by $Z^\bullet$. 
There are two spectral sequences for this double complex: we will use the first
one to compute $H^\*(Z)$ and the second one will be the one we are looking for.
From the first filtration on $Z^\bullet$ with
\[
  \tilde F^q\!~Z^p =
  \bigoplus_{\substack{r+s = p \\ s\geq q }} X^{r,s}
\]
we obtain a first spectral sequence converging to $H(Z^\*)$.
Its zeroth page $\tilde E_0$ is
\[
  \tilde E_0^{p,q}
  = \hom_U (Q_p, G(I^q))
\]
and its differential comes from the one on $Q_\*$. We claim that for each $s\geq
0$, the functor $\hom_U(-,G(I^s))$ is exact. Indeed, by the adjunction of
Proposition~\ref{prop:adjunction} it is naturally isomorphic to
$\hom_{U^e}(F(-),I^s)$, which is the composition of the functors
$F=U\otimes_S(-)$ and $\hom_{U^e}(-,I^s)$ and these are exact because $U$ is
left projective over $S$ and $I^s$ is $U^e$-injective. The first page $\tilde
E_1$ of the spectral sequence is therefore given~by
\[
  \tilde E_1^{p,q}
  =\begin{cases}
    \hom_{U}(N,G(I^q) )\cong\hom_{U^e}(F(N),I^q) &\text{if $p=0$;} \\
    0		&\text{if $p\neq 0$}
  \end{cases}
\]
and its differential is induced by that of $I^\*$. Now, as the complex
$\hom_{U^e}(F(N),I^\*)$ computes $\Ext_{U^e}^\*(F(N),M)$ using injectives, we
obtain that the second page is
\[
  \tilde E_2^{p,q}
  =\begin{cases}
    \Ext_{U^e}^q(F(N),M)	&\text{if $p=0$;} \\
    0				&\text{if $p\neq 0$.}
  \end{cases}
\]
This spectral sequence thus degenerates at its the second page, so that we see
that $H^\*(Z)$ is isomorphic to~$\Ext_{U^e}^\*(F(N),M)$.

The second filtration on $Z^\bullet$ is given by
\[
  F^p Z^q =
  \bigoplus_{\substack{r+s = q \\ r\geq p }} X^{r,s}
\]
and determines a second spectral sequence $E_\bullet$ that also converges to
$H(Z^\*)$.  Its  differential on $E_0$ is induced by the one on $I^\bullet$; as
$Q_p$ is $U$-projective for each $p\geq 0$, the cohomology of
$\hom_U(Q_p,G(I^\bullet))$ is given in its $q$th place precisely by $E_1^{p,q}
=\hom_U(Q_p,H^q(S,M))$ ---recall that, according to
Proposition~\ref{prop:injectives}, the cohomology of $G(I^\*)$ is $H^\*(S,M)$.
Since the differentials in $E_1$ are induced by those of $Q_\bullet$, for each
$q\geq 0$ the cohomology of the row~$ E_1^{\*,q}$ is $E_2^{p,q} 	=
\Ext_U^p(N,H^q(S,M))$.  The spectral sequence $E_\*$ is therefore the one we
were looking for.
\end{proof}

Specializing Theorem \ref{thm:spectral} to the case in which $N=S$
we obtain the
following corollary, which is in fact the result we are mainly interested in.

\begin{Corollary}\label{coro:spectral}
  If $L$ is $S$-projective then for each $U^e$-module $M$ there is a
  first-quadrant spectral sequence $E_\*$ converging to $H^\bullet(U,M)$ with
  second page
  \[
    E_2^{p,q}
    = H_S^p(L,H^q (S,M)).
  \]
\end{Corollary}

The following examples illustrate what happens when we take $M=U$ in the two
extreme situations.

\begin{Example}
  Suppose first that $L=0$. The enveloping algebra $U$ is just $S$ and
  $\Lambda^\*_SL=S$, so the resolution $U\otimes \Lambda_S^\bullet L$ of $S$ is
  simply $Q_\bullet =U\otimes_S S$. The double complex $X^{\*,\*}$ is therefore
  $\hom_S(S,\hom_{S^e}(S,I^\bullet))$, which is isomorphic to
  $\hom_{S^e}(S,I^\bullet)$ and the cohomology of the complex $Z^\*$ in the proof
  is $\HH^\*(S)$, the Hochschild cohomology of $S$.
\end{Example}

\begin{Example}
  If $S=\kk$ and $L=\g$ is a Lie algebra then
  $H^\bullet(S,U)=\Ext_{\kk^e}^\bullet(\kk,U)$ is just $U$, the second page of our
  spectral sequence is $H^\*(\g,U)$ and we recover from
  Corollary~\ref{coro:spectral} the well-known fact that the Hochschild cohomology
  of the enveloping algebra of a Lie algebra equals its Lie cohomology with values
  on $U$ with the adjoint action, as in~\cite{cartan-eilenberg}*{XIII.5.1}.
\end{Example}


\section{Eulerian modules}\label{sec:eulerian}

We assume from now on that $\kk$ is a field of characteristic zero.
In this section we pay attention to a particular but rather frequent situation
in which some calculations to attain the second page of the spectral sequence in
Corollary~\ref{coro:spectral} can be significantly shortened.
Let $S=\kk[x_1,\ldots,x_n]$.
The usual graded algebra structure on $S$,
such that $\abs{x_i}=1$ if $1\leq i\leq n$, induces a grading on the Lie
algebra~$\Der S$ that makes each partial derivative $\partial_i$
have degree $-1$. Let $L$ be a Lie subalgebra of $\Der S$ that is also an
$S$-submodule of $\Der S$ freely generated by 
homogeneous derivations $\alpha_1,\ldots,\alpha_l$, where
$\alpha_1=e=x_1\partial_1+\dots+x_n\partial_n$ is the \emph{eulerian}
derivation. 
The pair $(S,L)$ is a Lie-Rinehart algebra and, since $L$ is free, its
enveloping algebra $U$ admits the set
$\{\alpha_1^{n_1}\dots\alpha_l^{n_l}:n_1,\ldots,n_l\geq0\}$
as an $S$-module basis of $U$ thanks to the PBW-theorem
in~\cite{rinehart}*{\S3}.  The graded structures on $S$ and $\Der S$
induce on $L$ and on $U$ a graded Lie algebra and a graded
associative algebra structures. 

\begin{Definition}\label{def:eulerian}
  A $\ZZ$-graded left $U$-module $N=\bigoplus_{i\in\ZZ}N_i$
  is \emph{eulerian} if the
  action of $e$ on $N$ satisfies $e\cdot n=in$ if $n\in N_i$. 
\end{Definition}

\subsection{The Lie--Rinehart cohomology~\texorpdfstring{$H_S(L,N)$}{H\_(S,N)}}
Recall from Proposition~\ref{prop:ch-ei} that the Lie-Rinehart cohomology of
$(S,L)$ with values on an $U$-module $N$ is the cohomology of the complex 
$\Y{\*}=\hom_S(\Lambda^\*_SL,N)$
with differentials $d^r:\Y{r}\to\Y{r+1}$ determined by 
\[
  \begin{multlined}[.9\displaywidth]
    (d^r f)(\alpha_{i_1}\wedge\dots\wedge\alpha_{i_{r+1}}) 
    = \sum_{j=1}^{r+1}(-1)^{j+1}
    \alpha_{i_j} \cdot f(\alpha_{i_1}\wedge\dots\wedge\check\alpha_{i_j}
    \wedge\dots\wedge\alpha_{i_{r+1}})
    \\
    +\sum_{1\leq j < k\leq r+1} (-1)^{j+k} 
    f([\alpha_j,\alpha_k]  \wedge\alpha_{i_1}\dots\wedge\check\alpha_{i_j}
    \wedge\dots\wedge\check\alpha_{i_k}
    \wedge\dots\wedge\alpha_{i_{r+1}}),
  \end{multlined}
\]
with $f\in\hom_S(\Lambda^r_SL,N)$ and $1\leq i_1<\dots<i_{r+1}\leq l$ and
where $\check\alpha_{i}$ means that $\alpha_{i}$ has been omitted.  The
gradings on $S$, $L$ and $N$ induce a grading on each of the vector spaces in
the complex $\Y\*$ and the differentials are homogeneous with respect with
this grading, so that, if $\Y\*_i$ is the subcomplex of $\Y\*$ of degree~$i$,
there is a decomposition $\Y\*=\bigoplus_{i\in\ZZ}\Y\*_i$.  The cohomology of
$\Y\*$ is a graded complex: we write
$H_S^p(L,N)=\bigoplus_{i\in\ZZ}H_S^p(L,N)_i$, with
$H_S^p(L,N)_i=H^p(\Y{\*}_i)$ for each $p\geq0$.  The next proposition allows
us to see that $H_S(L,N)=H_S(L,N)_0$.

\begin{Proposition}\label{prop:rine-euler}
  Let $N$ be an eulerian $U$-module. The inclusion of the component of degree
  zero $\Y\*_0\hookrightarrow \Y\*$ is a quasi-isomorphism.
\end{Proposition}

\begin{proof}
  Let $\gamma^\*:\Y\*\to\Y\*$ be the linear map whose restriction to each
  homogeneous component of the complex $\Y\*$ is the multiplication by degree.
  A straightforward calculation shows that the 
  homotopy $s=(s_r :\Y{r}\to\Y{r-1})_{r\geq0}$ given by
  \(
  (s_rf)( \alpha_{i_1}\wedge\dots\wedge\alpha_{i_r})
  = f(e\wedge\alpha_{i_1}\wedge\dots\wedge\alpha_{i_r})
  \)
  satisfies $s\circ d +d\circ s = \gamma$. We obtain from this that  
  $\gamma$ induces the zero map in cohomology and then, as the
  field has characteristic zero, each of the cohomologies of the subcomplexes of
  nonzero degree are~trivial.
\end{proof}

\begin{Corollary}
  If $N$ is an eulerian $U$-module then the subspace
  $\bigcap_{i\geq2}\ker(\alpha_i:N_0\to N)$ of $N_0$ is isomorphic to
  $H_S^0(L,N)$. 
\end{Corollary}

\subsection{The Hochschild cohomology~\texorpdfstring{$H^\*(S,M)$}{H(S,M)}}
To compute the Hochschild cohomology of $S$ we use the Koszul resolution of
$S$ available in~\cite{weibel}*{\S 4.5}.

\begin{Lemma}\label{lem:koszul}
Let $W$ be the subspace of $S$ with basis $(x_1,\ldots,x_n)$.
The complex $P_\*=S^e\otimes\Lambda^\* W$
with differentials $b_\* :P_\*\to P_{\*-1}$
defined for $s, t\in S$ and $1\leq i_1<\dots<i_{r}\leq n$ by
\begin{align*}
  &b_r(s|t\otimes x_{i_1}\wedge \dots\wedge x_{i_r})
  = \sum_{j=1}^r(-1)^{j+1}
  (sx_{i_j}|t -s|x_{i_j}t) \otimes x_{i_1}\wedge\dots\wedge\check
  x_{i_j}\wedge\dots\wedge x_{i_r}
\end{align*}
and augmentation $\varepsilon : S^e\to S $ given by $\varepsilon(s|t) = st$
is a resolution of $S$ by free $S^e$-modules. Here the symbol 
$|$ denotes the tensor
product inside $S^e$ and $\check x_{i_j}$ means that $x_{i_j}$ is omitted.
\end{Lemma}

If $M$ is an graded $U^e$-module $M$, the cohomology of the complex
$\hom_{S^e}(P_\*,M)$ is $H^\bullet(S,M)$. The
graded algebra~$S$ induces a grading on this complex which is preserved
by the differentials and therefore $H^\*(S,M)$ inherits a graded structure.
We denote by $H^\*(S,M)_i$ the $i$th homogeneous component of $H^\*(S,M)$ for
each $i\in\ZZ$.

\begin{Proposition}\label{prop:E}
If $M=\bigoplus_{i\in\ZZ}M_i$ is a graded $U^e$-module such that
$(e\otimes1-1\otimes e)\cdot m =im$ for all $m\in M_i$  then for each
$q\in\ZZ$ the $q$th Hochschild cohomology space $H^q(S,M)$ is an eulerian
$U$-module.
\end{Proposition}

\begin{proof}
  That $H^q(S,U)_i$ is a graded $U$-module for each $i$
  can be seen from~\eqref{eq:U-structure}.
  Following~Remark~\ref{def:alphanabla} we denote by $e_S:S\to S $ the
  action of $e$ on $S$ and by $e_S^e$ the derivation
  $e_S\otimes 1+1\otimes e_S:S^e\to S^e$.
  We let, for $q\geq0$,
  $e_q:P_q\to P_q$ be the $e_S^e$-operator  such that 
  \[\label{eq:liftE}
    e_q(1|1\otimes x_{i_1}\wedge\dots\wedge x_{i_q}) 
    = q|1\otimes x_{i_1}\wedge\dots \wedge x_{i_q}
  \]
  if $1\leq i_1<\dots<i_{q}\leq n$.
  A small calculation allows us to deduce from~\eqref{eq:liftE} that the
  collection of maps $(e_q)_{q\geq0}$ is a $e^e_S$-lifting of $e_S$ to
  $P_\*$. Let now $\phi\in\hom_{S^e}(P_q,M)$ be an homogeneous map of
  degree $i$ and write $m_{i_1,\ldots,i_q}\coloneqq\phi(1|1\otimes
  x_{i_1}\wedge\dots x_{i_q})\in M_{i+q}$. Our hypothesis on $M$ allows us to
  see that 
  \begin{align*}
    \MoveEqLeft
    e_q^\sharp(\phi)(1|1\otimes x_{i_1}\wedge\dots\wedge x_{i_q})\\
    &= (e\otimes 1-1\otimes e)\cdot m_{i_1,\ldots,i_q}
    -\phi\circ e_q(1|1\otimes x_{i_1}\wedge\dots\wedge x_{i_q}) \\
    &=(i+q)m_{i_1,\ldots,i_q}-qm_{i_1,\ldots,i_q}
    =im_{i_1,\ldots,i_q}
  \end{align*}
  and therefore $\nabla_e^q([\phi])=i[\phi]$. 
\end{proof}

\section{The algebra of differential operators tangent to a central
arrangement of three lines}\label{sec:ha}

In this section we describe the example that motivated us to construct the
spectral sequence of Corollary~\ref{coro:spectral}: it is the algebra of
differential operators $\Diff\A$ tangent to a central arrangement of lines
$\A$, whose Hochschild cohomology was studied by the first author and
M.\,Suárez-Álvarez in~\cite{koma}.
We will regard $\Diff\A$ as the enveloping algebra of a Lie--Rinehart algebra
and compute the second page
$E_2^{p,q}=H_S^p(L,H^q(S,U))$ of the spectral sequence of
Corollary~\ref{coro:spectral} for a
central line arrangement of three lines. After
studying the Lie-Rinehart cohomology in a generic situation, we will compute
what we need of $H^\*(S,U)$ and the action of $U$ to obtain the
second page and, finally, the Hochschild cohomology $\HH^\*(U)$
in Corollary~\ref{coro:HHDiffA}.

Let $S=\kk[x,y]$ and write
the defining polynomial of the arrangement $Q=xF$ with
$F=y(tx+y)$, for some $t\in\kk$.  
H.\,Saito's criterion~\cite{saito}*{Theorem 1.8.ii}
allows us to see that the two derivations
\begin{align*}
  &E=x\partial_x + y\partial_y,
  && D = F\partial_y
\end{align*}
form an $S$-basis of $\Der\A$. In \cite{koma} there is a convenient
presentation of~$U=\Diff\A$. It is generated by the symbols $x$, $y$, $D$ and
$E$ subject to the relations
\begin{align}
  & [y,x] = 0, \\
  & [D,x] = 0, && [D,y] = F, \\
  & [E,x] = x, && [E,y] = y, && [E,D] = D,
\end{align}
where the bracket $[a,b]$ between two elements stands for the commutator
$ab-ba$. Moreover, the set $\{x^{i_1}y^{i_2}D^{i_3}E^{i_4} :
i_1,\ldots,i_4\geq0 \}$ is a basis of $U$ as a vector space.

As in Section~\ref{sec:eulerian},
we view $S$ as a graded algebra, with both~$x$ and~$y$ of degree~$1$,
and for each $i\geq0$ we write $S_i$ the homogeneous component of~$S$ of
degree~$i$. This grading induces one in $L\coloneqq \Der\A$ and also
on $U$:

\begin{Proposition}\label{prop:Ugrading}
There is a grading on the algebra $U$ with $\abs{x}=\abs{y}=\abs{D}=1$ and
$\abs{E}=0$. Given $i\geq0$ the $i$th homogeneous component $U_i$ of $U$ is
the right $\kk[E]$-module generated by the set $\{x^ry^sD^t : r+s+t = i\}$.
\end{Proposition}

For convenience, we denote by $\psi'$ the image of $\psi\in\kk[E]$ under the
linear map $\kk[E]\to \kk[E]$ such that $E^n\mapsto E^n - (E+1)^n$ for every
$n\geq0$. Recall that $\otimes$ or $|$ denote the tensor product over $\kk$
and that we may sometimes omit it to alleviate notation.

\subsection{The Lie-Rinehart
cohomology~\texorpdfstring{$H^\*_S(L,N)$}{H\_S(L,N)}}
\label{subsec:lrdiff}
We let $V_L$ be the subspace of $L$ with basis $(D,E)$ and $V_L^*$ be its dual
space, and denote the dual basis by $(\DD,\EE)$. Let $N$ be an eulerian
$U$-module.
The Lie-Rinehart
cohomology $H_S^\*(L,N)$ of $(S,L)$ with values on $N$ is the cohomology of the
complex $C^\*_S(L,N)$, which is
isomorphic via standard identifications
to the complex~$N\otimes \Lambda^\*V_L^*$ given by
\[ \label{eq:rine-N}
  \begin{tikzcd}[column sep=1.75em]
    N \arrow[r, "d^0"]
    & N\otimes V_L^* \arrow[r, "d^1"]
    & N\otimes \Lambda^2V_L^*
  \end{tikzcd}
\]
with differentials
\begin{align}
  &d^0(n) = D\cdot n\otimes \DD + E\cdot n\otimes\EE ;\\
  &d^1(n\otimes\DD + m\otimes\EE) 
  = \paren{D\cdot m - E\cdot n + n}\otimes\DD\wedge\EE.
\end{align}

\begin{Proposition}\label{prop:rine-N}
  Let $N=\bigoplus_{i\in\ZZ}N_i$ be an eulerian $U$-module and
  $\nabla_D:N_0\to N_1$ be the restriction of the action of $D$.
  There are isomorphisms of vector spaces
  \[
    H^p_S(L,N) \cong
    \begin{cases*}
      \ker\nabla_D, &if $p=0$;\\
      \coker  \nabla_D \otimes \kk\DD
      ~\oplus ~ \ker \nabla_D \otimes\kk\EE, 
      &if $p=1$;\\
      \coker\nabla_D\otimes\kk\DD\wedge\EE, 
      &if $p=2$
    \end{cases*}
  \]
  and $H^p_S(L,N)= 0 $ for every other $p\in\ZZ$.
\end{Proposition}

We
notice that the cohomology $H_S^\*(L,N)$ depends only on the map $N_0\to N_1$
given by multiplication by $D$.

\begin{proof}
Thanks to Proposition~\ref{prop:rine-euler}, we need only compute the
cohomology of the subcomplex of $N\otimes \Lambda^\*V_L^*$ of degree zero.
This subcomplex is
\[
  \begin{tikzcd}[column sep=1.75em]
    N_0 \arrow[r, "d_0^0"]
    & N_1\otimes\kk \DD\oplus N_0\otimes \kk\EE \arrow[r, "d_0^1"]
    & N\otimes \kk\DD\wedge\EE
  \end{tikzcd}
\]
with differentials given by $d_0^0(n) = D\cdot n\otimes \DD$ and  
\(
d_0^1(n\otimes\DD + m\otimes\EE) 
  = D\cdot m\otimes\DD\wedge\EE.
\)
The claim in the proposition follows immediately from the these expressions.
\end{proof}

In Proposition~\ref{prop:E} we saw that the $U$-modules $H^\*(S,U)$
are eulerian and, as a consequence
of this, to get $H_S^\*(L,H^\bullet(S,U))$ we may use the following strategy:
to compute the homogeneous components of degree $0$~and~$1$ of
$H^\bullet(S,U)$ and then to describe the map $\nabla_D^\* :
H^\bullet(S,U)_0\to H^\bullet(S,U)_1$ given by the action~of~$D$.


\subsection{The Hochschild cohomology \texorpdfstring{$H^\*(S,U)$}{H(S,U)}}
Let $W$ be the subspace of $S$ with basis $(x,y)$.
Applying $\hom_{S^e}(-,U)$ to the Koszul resolution in Lemma~\ref{lem:koszul}
and using standard identifications we obtain the complex 
\[ \label{eq:ext-SU}
  \begin{tikzcd}[column sep=1.75em]
    U \arrow[r, "\delta^0"]
    & U\otimes \hom\paren{W,\kk} \arrow[r, "\delta ^1"]
    & U\otimes \hom(\Lambda^2W, \kk)
  \end{tikzcd}
\]
with differentials
\begin{align}
  &\delta^0(u) = [x,u]\xx + [y,u]\yy \\
  &\delta^1(a\xx + b\yy)
  =\left(  [x,b] - [y,a] \right)\xx\wedge\yy,
\end{align}
where $(\xx,\yy)$ is the dual basis of $(x,y)$ and $\xx\wedge\yy$ is the
linear morphism $\Lambda^2W \to\kk$ that sends $x\wedge y$ to one.  The
cohomology of the complex~\eqref{eq:ext-SU} is $H^\bullet(S,U)$.

\begin{Proposition}\label{prop:h2SU}
There are isomorphisms of graded vector spaces
$
H^0(S,U)\cong  S$ and $
H^2(S,U)\cong \kk[D]\otimes\kk[E]\otimes \kk(\xx\wedge\yy).$
\end{Proposition}

\begin{proof}
Evidently, $H^0(S,U)$, the subset of $U$ of elements that
commute with $x$ and $y$, contains $S$: let us prove
that they are equal. Given $u\in H^0(S,U)$, there exist $v_0,\ldots,v_m$ in
the subalgebra of $U$ generated by $x$,~$y$ and~$D$ such that
$u=\sum_{i=0}^mv_iE^i$.  The condition $0=[u,x]$ implies that
$0=\sum_{i=0}^mv_i(E^i)'$ and therefore that $v_i=0$ for every $i>0$, so that
there exist $f_1,\ldots,f_n\in S $ such that $u=\sum_{i=0}^nf_iD^i$. An
inductive argument using that
\[
  0=[u,y]
  =\sum_{i=0}^nf_i[D^i,y]
  \equiv nf_nD^{n-1} \mod\bigoplus_{i=0}^{n-2}SD^i
\]
allows us to see that $f_i=0$ if $i>1$ and therefore to conclude that $u\in
S$.

We compute $H^2(S,U)$ directly from the complex~\eqref{eq:ext-SU}. Denote by
$S_{\geq 1}$ the space of polynomials with no constant term. We claim that
$S_{\geq1}D^k\kk[E]$ is contained in the image of $\delta^1$ for every
$k\geq0$. Indeed, if $f,g\in S$ and $\psi \in \kk[E]$ then 
\[
  \delta^1(g\varphi\xx + f\psi\yy )
  = (xf\psi'-yg\varphi')\xx\wedge\yy,
\]
so that our claim is true if $k=0$.  Assume now that $k>0$ and that for every
$j< k$ the inclusion $S_{\geq 1}D^j\kk[E]\subset \im\delta^1$ holds. Given
$f\in S$ and $\psi\in \kk[E]$, we have that
\begin{align}
  \delta^1(fD^k\psi\yy) 
  &= xfD^k\psi'\xx\wedge\yy \\
  \shortintertext{and}
  \delta^1(fD^k\psi\xx) 
  &=(-f[y,D^k]\psi -  fD^ky\psi' )\xx\wedge\yy	\\
  &=(- f[y,D^k](\psi-\psi') -  fyD^k\psi')\xx\wedge\yy  \\
  &\equiv - fyD^k\psi'\xx\wedge\yy \mod \im \delta^1,
\end{align}
which proves the claim. We easily see, on the other hand, that the
intersection of $\kk[D]\kk[E]$ with $\im\delta^1$ is trivial, so that
$H^2(S,U)\cong \kk[D]\kk[E]\xx\wedge\yy$, as we wanted.
\end{proof}

The computation of $H^1(S,U)$ is significantly more involved than the one just
above. As we are after the Lie--Rinehart cohomology $H_S^\*(L,H^1(S,U))$,
thanks to Proposition~\ref{prop:rine-N} we need only compute the homogeneous
components of $H^1(S,U)$ of degree~$0$~and~$1$.

\begin{Proposition}\label{prop:h1SU}
The graded vector space $H^1(S,U)$ satisfies $\dim H^1(S,U)_0=5$ and
$\dim H^1(S,U)_1=8$.
Moreover, $H^1(S,U)_0$ is generated by the classes of the cocycles
of the complex~\eqref{eq:ext-SU}
\begin{align*}
  &\eta_1 =  (-yE+D)\xx + tyE\yy, 
  &&\eta_2 = y\xx, 
  &&\eta_3 = x\yy, 
  &&\eta_4 = y\yy,
  &&\eta_5 = D\yy,
\end{align*} 
and
$H^1(S,U)_1$ is generated by the classes of the cocycles
\begin{align*}
  &\zeta_1 = (D^2 - 2yDE + y^2(E^2-E) )\xx 
  + (2tyDE + tFE +ty^2(E-E^2)) \yy,\\
  &\zeta_2 = (-y^2E + yD)\xx + ty^2E\yy,
  \quad\zeta_3 = y^2\xx, 
  \qquad\zeta_4=x^2\yy, 
  \\
  &\zeta_5=xy\yy,\qquad\zeta_6=xD\yy,
  \qquad\zeta_7 = yD\yy,
  \qquad\zeta_8 = D^2\yy.
\end{align*} 
\end{Proposition}

\begin{proof}
The homogeneous component of degree zero of the complex~\eqref{eq:ext-SU} is
\[
  \!\begin{tikzcd}[column sep=1.75em]
    U_0 \arrow[r, "\delta^0_0"]
    & U_1\xx\oplus U_1\yy \arrow[r, "\delta_0^1"]
    & U_2\xx\wedge\yy
  \end{tikzcd}
\] 
with $U_0=\kk[E]$, $U_1=S_1\kk[E]\oplus D\kk[E]$, 
\[\label{eq:U2}
  U_2
  =S_2\kk[E]\oplus S_1D\kk[E]\oplus D^2\kk[E]
\]
and differentials given by
\begin{align}
  &\delta^0_0(\phi)
  =x\phi'\xx + y\phi'\yy, \\
  &\delta^1_0\left(  (x\varphi_1+y\varphi_2 + D\varphi_3)\xx \right)
  =\left( -xy\varphi_1'-y^2\varphi_2' - yD\varphi_3' 
  - F(\varphi'_3-\varphi_3)\right) \xx\wedge\yy, \\
  &\delta^1_0\left( (x\psi_1+y\psi_2 + D\psi_3)\yy \right)
  =( x^2\psi_1' + xy\psi'_2 + xD\psi_3')\xx\wedge\yy,
\end{align}
where $\phi$, $\varphi$'s and $\psi$'s denote elements of $\kk[E]$.  

Let $a,b\in U_1$ and let $\omega=a\xx+b\yy$ be a $1$-cocycle. Up to adding a
coboundary we may suppose that the component of $a$ in $x\kk[E]$ is zero: we
may therefore write
\begin{align}\label{eq:h10}
  &a 
  = y\varphi_2 + D\varphi_3,
  &&b
  = x\psi_1 +y\psi_2 + D\psi_3,
\end{align}
with Greek letters in $\kk[E]$. The coboundary $\delta_0^1(\omega)$ belongs to
$U_2\xx\wedge\yy$, which decomposes as in~\eqref{eq:U2}.  The vanishing of the
component in $D^2\kk[E]$ does not give any information, that of the one in
$S_1D\kk[E]$ tells us that $\varphi_3'=\psi_3'=0$ and, finally, that of
$S_2\kk[E]$ tells us~that
\[\label{eq:H^1_0}
  x^2\psi_1' + xy\psi_2' 
  = y^2\varphi_2' - F\varphi_3'.
\]
Let us put $\lambda\coloneqq\varphi_3$.  Looking at the component on
$y^2\kk[E]$ of equation~\eqref{eq:H^1_0} and keeping in mind that $F=y^2+txy$
we see that $\varphi_2'=\lambda$ and, using this, that
$x\psi_1'+y\psi_2'=-\lambda ty$. There exist then $\mu\in\kk$ and $f_1\in S_1$
such that
\begin{align}
  &\varphi_2
  = -\lambda E+\mu,
  &&x\psi_1+y\psi_2=\lambda tyE + f.
\end{align}
As a cocycle $\omega=a\xx+v\yy$ satisfying~\eqref{eq:h10} is a coboundary only
if it is zero, we conclude that  
\(
  H^1(S,U)_0
  \cong \kk\eta_1\oplus\kk y\xx\oplus (S_1\oplus\kk D)\yy,
\)
with $\eta_1 = (-yE+D)\xx + tyE\yy$.

\bigskip

We now compute $H^1(S,U)_1$. The component of degree~$1$ of
the~complex~\eqref{eq:ext-SU} is
\[\label{eq:complex1}
  \begin{tikzcd}[column sep=1.75em]
    U_1 \arrow[r, "\delta^1_0"]
    & U_2\xx\oplus U_2\yy \arrow[r, "\delta_1^1"]
    & U_3\xx\wedge\yy
  \end{tikzcd}
\] 
with $U_3 = S_3\kk[E]\oplus S_2D\kk[E]\oplus S_1D^2\kk[E]\oplus D^3\kk[E]$
and differentials 
\begin{align*}
  &\delta_1^0 ( x\phi_1 + y\phi_2 + D\rho ) \\
  &\qquad =
  (x^2\phi_1' + xy\phi_2' + xD\rho')\xx 
  +(xy\phi_1' + y^2\phi_2' + yD\rho' + F(\rho'-\rho) \yy,\\
  &\delta_1^1	\left(\left(
  \sum x^iy^j\varphi_{ij} + xD\varphi_1 + yD\varphi_2 +D^2\varphi
  \right) \xx\right) \\
  &\qquad= -\sum x^iy^{j+1}\varphi_{ij}' 
  - xyD\varphi_1' - xF(\varphi_1'-\varphi_1) 	
  -y^2D\varphi_2' - yF(\varphi_2'-\varphi_2) \\
  &\qquad\hphantom{{}={}}
  -yD^2\varphi' -2FD(\varphi_2'-\varphi_2) - FF_y(\varphi'-\varphi), \\
  &\delta_1^1	\left(\left(
  \sum x^iy^j\psi_{ij} + xD\psi_1 + yD\psi_2 +D^2\psi
  \right)\yy\right) \\
  &\qquad=
  \sum x^{i+1}y^j\psi'_{ij}+x^2D\psi_1' + xyD\psi_2' + xD^2\psi'.	
\end{align*}
In all the sums that appear here the indices $i$ and $j$ are such that $i+j=2$
and we have omitted the factor $\xx\wedge\yy$ for~$\delta_1^1$.
Again, all Greek letters lie in~$\kk[E]$.

Let us put, once again, $\omega = a\xx+b\yy$, this time with $a$ and $b$ in
$U_2$. Up to coboundaries, we write, with the same conventions as before,
\begin{align*}
  &a
  = y^2\varphi_{02} + yD\varphi_2 + D^2\varphi,
  &b
  =\sum x^iy^j\psi_{ij} + xD\psi_1 + yD\psi_2 + D^2\psi.
\end{align*}
Let us examine the condition
$\delta_1^1(\omega)=0$ component by component according to our description of
$U_2$ in~\eqref{eq:U2} above.

In $D^3\kk[E]$ there is no condition at all.
In $S_1D^2\kk[E]$ we have $xD^2\psi'-yD^2\varphi'=0$, so that $\psi$ and
$\varphi$ are scalars.
In $S_2D\kk[E]$ the condition reads
\[\label{eq:S2DT}
  x^2D\psi_1' +xyD\psi_2'
  =y^2D\varphi_2' + 2FD(\varphi'-\varphi).
\]
Writing $F=y^2+txy$ and looking at the terms that are in $y^2\kk[E]$ 
we find $0=\varphi_2'-2\varphi$, and
then $\varphi_2=-2\varphi E+\lambda$ for some $\lambda\in\kk$. What 
remains of \eqref{eq:S2DT} implies that $x\psi_1'+y\psi_2'
= -2ty\varphi$ and therefore there exists $h\in S_1$ such that
\[
  xD\psi_1+yD\psi_2
  = 2\varphi ty DE + h D.
\]

Finally, we look at $S_3\kk[E]$: we have
\[
  \sum x^{i+1}y^j\psi_{ij}'
  = y^3\varphi_{02}' + yF(\varphi'_2-\varphi_2) - FF_y\varphi.
\]
In particular, using that $F_y = 2y+tx$ and looking at the terms in $y^3\kk[E]$, we
find that $0=\varphi_{02}' + (\varphi_2'-\varphi_2) +2 (\varphi'-\varphi)$,
or, rearranging, $\varphi_{02}'=-2\varphi E+ \lambda$. ``Integrating'', we see
there exists~$\mu\in\kk$ such that
\[
  \varphi_{02} 
  = \varphi(E^2-E)-\lambda E+\mu.
\]
Now, as $FF_y = 2y^3 + 3txy^2 + t^2x^2y$, we must have
\[
  \sum x^iy^j\psi_{ij}'
  =ty^2(\varphi_2'-\varphi_2) - (3ty^2 + t^2xy)\varphi,
\]
and, integrating yet another time, we get
$\sum x^iy^j\psi_{ij}= \phi(tFE + ty^2(E-E^2)) + \lambda ty^2E $,

We conclude in this way that every
$1$-cocycle of degree~$1$ is cohomologous to one of the form
\[\label{eq:H1(S,U)1}
  \omega
  =\varphi\zeta_1 + \lambda\zeta_2 
  + f\yy + hD\yy + \psi D^2\yy + \mu y^2\xx
\]
where $\zeta_1$ and $\zeta_2$ are the cocycles in the statement,
$\varphi$, $\lambda$, $\psi$, $\mu\in\kk$,
$h\in S_1$ and $f\in S_2$.

It is easy to see from the expression we have for
$\delta_1^0$ that such a cocycle is a coboundary if and only if it is a scalar
multiple of $F\yy$. The upshot of all this is that
\[
  H^1(S,U)_1
  \cong \lin{\zeta_1,\zeta_2}\oplus 
  \kk y^2\xx \oplus
  \left( S_2/(F)\oplus S_1D\oplus\kk D^2\right) \yy,
\]
as we wanted.
\end{proof}

\subsection{The action of \texorpdfstring{$U$ on $H^\*(S,U)$}{U on H(S,U)}}
\label{subsec:UonHdiff}

As we have already computed in Propositions~\ref{prop:h2SU}
and~\ref{prop:h1SU} the homogeneous components of degrees~$0$ and~$1$ of the
Hochschild cohomology $H^q(S,U)$ for each $q$, Proposition~\ref{prop:rine-N}
tells us that in order 
to compute the second page~$E_2^{\*,q}=H_S^\*(L,H^q(S,U))$ it
remains only to find the kernel and the cokernel of~$\nabla_D^q:H^q(S,U)_0\to
H^q(S,U)_1$.

\begin{Proposition}\label{prop:D}
  \begin{thmlist}
  \item\label{prop:nablaD0}
    The kernel of $\nabla_D^0:H^0(S,U)_0\to H^0(S,U)_1$ is $\kk$ and its
    cokernel is $S_1$, the subspace of $S$ with basis $(x,y)$.

  \item\label{prop:nablaD2}
    The kernel of $\nabla_D^2:H^2(S,U)_0\to H^2(S,U)_1$ is $\kk
    D^2\xx\wedge\yy$ and its cokernel is zero.

  \item\label{prop:nablaD1}
    The map $\nabla_D^1 : H^1(S,U)_0\to H^1(S,U)_1$ is a monomorphism and its
    cokernel is generated by the classes of the cocycles
    $\zeta_1$, $\zeta_6 $, and $\zeta_8$ given in
    Proposition~\ref{prop:h1SU}.
  \end{thmlist}
\end{Proposition}

\begin{proof}
Recall that $H^\*(S,U)$ is computed from the Koszul resolution $P_\*$
of~Lemma~\ref{lem:koszul}, where $W$ is the vector space spanned by $x$ and
$y$.  To describe the action of $D$ on $H^\*(S,U)$ we need a lifting of
$D_S:S\to S$ to an $P_\*$. We obtain one by letting, for each $q\in\{0,1,2\}$,
$D_q:P_q\to P_q$
be the $D^e_S$-operator such that
\[\label{eq:liftD}
  \begin{aligned}
    &D_0(1|1) = 0, \\
    & D_1(1|1\otimes y) = (1|y+y|1+tx|1)\otimes y + t|y\otimes x,
    &&D_1(1|1\otimes x)=0,\\
    &D_2(1|1\otimes x\wedge y ) = (1|y+y|1+tx|1)\otimes x\wedge y,
  \end{aligned}
\]
as a straightforward calculation shows.
From the description of $D_0$ we see that
the restriction to $S_0\to S_1$ of the map $\nabla_D^0:S\to S$ is zero, thus
proving assertion~\ref{prop:nablaD0}.

We recall from Proposition~\ref{prop:h2SU} that the homogeneous components of
degree~$0$ and~$1$ of $H^2(S,U)$ are $D^2\kk[E]\xx\wedge\yy$ and
$D^3\kk[E]\xx\wedge\yy$, respectively.  Let us  compute the kernel and the
cokernel of $\nabla_D^2:H^2(S,U)_0\to H^2(S,U)_1$. We have
\[
  D_2^\sharp (D^2\varphi\xx\wedge\yy)
  =\left(
  [D,D^2\varphi] - 
  D^2\varphi\xx\wedge\yy\left(D_2(1|1\otimes x\wedge y)\right)
  \right)\xx\wedge\yy
\]
and, as in the second term there never appears a higher power of $D$ than
$D^2$,
\[
  D_2^\sharp (D^2\varphi\xx\wedge\yy)
  \equiv D^3\varphi'\xx\wedge\yy \mod \im\delta_1^1.
\]
The claim in the second item follows from this.

For~\ref{prop:nablaD1} we give explicit formulas for the evaluation of
$\nabla_D^1:H^1(S,U)_0\to H^1(S,U)_1$ and, at the same time, compute its
cokernel. Suppose that $\omega$ is a representative of a class in $H^1(S,U)$
chosen as in~\eqref{eq:H1(S,U)1}.  As $D_1^\sharp(D\yy)=(-F_yD-F)\yy$, we see
that up to adding to $\omega$ an element in the image of $\nabla_D^1$ we may
suppose that $h=h_0x$, for some $h_0\in\kk$.

Let $\alpha$, $\beta$ and $\gamma$ in $\kk$ and define $\phi=\alpha y \xx +
(\beta x+\gamma y)\yy$. Since $\phi(D_1(1|1\otimes y))$ is equal to $ \gamma x
F_x - \alpha y F_x - \beta x F_y$, we have
\begin{align}
  D_1^\sharp(\phi) 
  &= \left( [D,\alpha y]-\phi(D_1(1|x|1))  \right)\xx
  + \left( [D,\beta x+\gamma y] - \phi(D_1(1|y|1))  \right)\yy \\
  &= \alpha F \xx
  +( \gamma yF_y + \alpha y F_x + \beta xF_y)\yy.
\end{align}
In view of this, it is easy to see that we may choose $\alpha$, $\beta$ and
$\gamma$ in such a way that $\omega+ D_1^\sharp(\phi)$, which is a cocycle of
the form~\eqref{eq:H1(S,U)1}, has $\mu=0$ and $f=0$ since 
$\{yF_x,xF_y,F\}$~spans~$S_2$.

Let us see that the $1$-cocycle~$\zeta_2$ belongs to the image of
$\nabla_D^1$.  Using the $1$-cocycle $\eta_1=(-yE+D)\xx + tyE\yy$ we get
\begin{align*}
  &D_1^\sharp(\eta_1)(1|1\otimes x)
  = [D,-yE + D] = -FE + yD, \\
  \shortintertext{and}
  &D_1^\sharp(\eta_1)(1|1\otimes y)
  = [D,tyE] - \eta_1(D_1(1|1\otimes y)) \\
  &\quad 
  = tFE - tyD -t(-yE +D) y - (tx +y) tyE - tyEy \\
  &\quad
  =  -2tyD  +ty^2 +t(y^2+txy),
\end{align*}
which belongs to $S_2 + \kk yD$.  We already know that the elements of $\left(
S_2 + \kk yD \right)\yy$ are coboundaries: it follows that
$D_1^\sharp(\eta_1)\equiv (-FE +yD)\xx$ modulo coboundaries.  Now, the
difference between $D_1^\sharp(\eta_1)$ and $\zeta_2$ is cohomologous to
$txyE\xx +ty^2E\yy$, which is in turn equal to $\delta_1^0(-tyE)$. As a
consequence of this, we have that $\nabla_D^1(\eta_1)$ is equal to $\zeta_2$
in cohomology.

We conclude from the preceding calculation that 
\(
\coker\left(\nabla_D^1 : H^1(S,U)_0\to H^1(S,U)_1\right)
\)
is generated by the classes of $\zeta_1$, $xD\yy$, and $D^2\yy$.
Since these classes are linearly independent, the dimension of this cokernel is
$3$. Finally, we can use the dimension theorem  to see that 
$\nabla_D^1 : H^1(S,U)_0\to H^1(S,U)_1$ is a monomorphism.
\end{proof}

\subsection{The second page}
\label{subsec:E2}
We have already made all the computations required for the second page of the
spectral sequence.

\begin{Proposition}\label{prop:dim2page}
The second page of the spectral sequence $E_\*$ of
Corollary~\ref{coro:spectral} converging to the Hochschild cohomology
$\HH^\*(U)$ has dimensions
\[\label{eq:2ndpage}
  \dim E_2^{p,q}=
  \begin{tikzpicture}[baseline={([yshift=-.8ex]current bounding box.center)}]
    \matrix (m) [matrix of math nodes,
      nodes in empty cells,nodes={minimum width=5ex,
      minimum height=5ex,outer sep=-5pt},
    column sep=1ex,row sep=-2ex]{ 
      q     &      &     &     &  \\
      &  1   & 1 &  0  & \\
      &  0   & 3	& 3  & \\
      & 1 & 3		& 2 \\
    \quad\strut &     &    &    & p \\};
    \draw[-Latex] (m-5-1.east) -- (m-1-1.east) ;
    \draw[-Latex] (m-5-1.north) -- (m-5-5.north) ;
  \end{tikzpicture}
\]
\end{Proposition}

\begin{proof}
Let $q\geq0$ and recall that the $q$th row $E_2^{\*,q}$ is equal to the
Lie-Rinehart cohomology $H^\*_S(L,H^q(S,U))$.  Thanks to
Proposition~\ref{prop:E}, $H^q(S,U)$ is an eulerian $U$-module and we may
use Proposition~\ref{prop:rine-N}, which asserts that to obtain
$H^\*_S(L,H^q(S,U))$ we need only the nullity and rank of
$\nabla_D^q:H^q(S,U)_0\to H^q(S,U)_1$.  This information is provided by
Proposition~\ref{prop:D}.
\end{proof}

\begin{Corollary}\label{coro:hh3}
  The dimension of $\HH^3(U)$ is $3$ or $4$.
\end{Corollary}

\begin{proof}
  The differential in the second page~\eqref{eq:2ndpage}
  could be non-zero, since
  neither the domain nor the codomain of the map $d_2^{0,2}:E_2^{0,2}\to
  E_2^{2,1}$ are.  As $\dim E_2^{0,2} = 1$, the differential $d_2^{0,2}$ is
  either 
  zero or a monomorphism. If it is zero, the sequence degenerates and using
  Corollary~\ref{coro:spectral} we obtain that $\dim
  \HH^3(U)=4$; if not, we have $\dim\HH^3(U)=3$.  
\end{proof}

It follows from Corollary~\ref{coro:hh3} that to see whether the sequence
degenerates or not it is enough to compute the dimension of~$\HH^3(U)$: this
provided in the next~proposition.

\begin{Proposition}\label{prop:hh3}
  The dimension of $\HH^3(U)$ is at least $4$.
\end{Proposition}

\begin{proof}
The Hochschild cohomology of the algebra of differential operators $U$ on an
arrangement of more than five lines is computed in~\cite{koma} from a
complex that we may still use. This complex is given by
$U\otimes\Lambda^\* V_U^*$, where $V_U$ is the subspace of~$U$ spanned
by $x$,~$y$,~$D$ and~$E$, or more graphically
\[
  \begin{tikzcd} 
    U \arrow[r, "d^0"]
    & U\otimes V_U^* \arrow[r, "d^1"]
    & U\otimes \Lambda^2V_U^* \arrow[r, "d^2"]
    & U\otimes \Lambda^3V_U^* \arrow[r, "d^2"]
    & U\otimes \Lambda^4V_U^*,
  \end{tikzcd}
\]
with differentials such that
\begin{gather*}
  d^2(u\otimes \hat x\wedge\hat y)
  = ([D,u] - \nabla_y^u(F))\otimes \hat x\wedge\hat y\wedge\hat D
  + ([E,u]-2u)\otimes \hat x\wedge\hat y\wedge\hat E;
  \\
  d^2(u\otimes \hat x\wedge\hat E)
  = -[y,u]\otimes \hat x\wedge\hat y\wedge\hat E
  -[D,u]\otimes \hat x\wedge\hat D\wedge\hat E
  +tuy\otimes \hat y\wedge\hat D\wedge\hat E;
  \\
  d^2(u\otimes \hat y\wedge\hat E)
  = [x,u]\otimes \hat x\wedge\hat y\wedge\hat E
  +((tx+2y)u-[y,u]-[D,u])\otimes \hat y\wedge\hat D\wedge\hat E;
  \\
  d^2(u\otimes \hat x\wedge\hat D)
  = -[y,u]\otimes \hat x\wedge\hat y\wedge\hat D
  +([E,u]-2u)\otimes \hat x\wedge\hat D\wedge\hat E;
  \\
  d^2(u\otimes \hat y\wedge\hat D)
  = [x,u]\otimes \hat x\wedge\hat y\wedge\hat D
  +([E,u]-2u)\otimes \hat y\wedge\hat D\wedge\hat E;
  \\
  d^2(u\otimes \hat D\wedge\hat E)
  = [x,u]\otimes \hat x\wedge\hat D\wedge\hat E
  + [y,u]\otimes \hat y\wedge\hat D\wedge\hat E;
  \\[5pt]
  d^3(u\otimes \hat x\wedge\hat y\wedge\hat D)
  = (-[E,u]+3u)\otimes\xx\wedge\yy\wedge\DD\wedge\EE
  \\
  d^3(u\otimes \hat x\wedge\hat y\wedge\hat E)
  =([D,u]-(tx+2y)u+[y,u])\otimes \hat x\wedge\hat y\wedge\hat D\wedge\hat E;
  \\
  d^3(u\otimes \hat x\wedge\hat D\wedge\hat E)
  = -[y,u]\otimes \hat x\wedge\hat y\wedge\hat D\wedge\hat E;
  \\
  d^3(u\otimes \hat y\wedge\hat D\wedge\hat E)
  = [x,u]\otimes \hat x\wedge\hat y\wedge\hat D\wedge\hat E.
\end{gather*}

As~$V_U$ is a homogeneous subspace of~$U$, the grading of $U$ induces on the
exterior algebra~$\Lambda^\* V_U$ an internal grading. There is as well a
natural internal grading on the complex~$U\otimes\Lambda^\* V_U^*$ coming from
the grading of $U$, with respect to which the differentials are homogeneous.
Moreover, the inclusion $\X^\*=(U\otimes\Lambda^\* V_U^*)_0\hookrightarrow
U\otimes\Lambda^\* V_U^*$ of the component of degree zero of the complex
$U\otimes\Lambda^\* V_U^*$ is a quasi-isomorphism: we will use the complex
$\X^\*$ again to compute $\HH^3(U)$. 

We borrow from our previous calculations the following four
cochains~in~$\X^3$: 
\begin{align*}
  &\!\begin{multlined}[.9\displaywidth]
    \omega_1 
    = D^2\xx\wedge\yy\wedge\EE
    +\left( 2D^2E-2yDE^2+F(E^3-2E^2+E)/2 \right)
    \otimes\yy\wedge\DD\wedge\EE \\
    +\left( -tD^2E + 2tyDE^2 + tfE^2\right)\otimes\yy\wedge\DD\wedge\EE,
  \end{multlined}
  \\
  &\!\begin{multlined}[.9\displaywidth]
    \omega_2 
    = \left( D^2 - 2yDE + y^2(E^2-E)  \right)\otimes\xx\wedge\DD\wedge\EE\\
    + \left(  2tyDE + tFE +ty^2(E-E^2) \right)
    \otimes\yy\wedge\DD\wedge\EE,
  \end{multlined}
  \\
  &\omega_3 = D^2\otimes\yy\wedge\DD\wedge\EE, \\
  &\omega_4 = xD\otimes\yy\wedge\DD\wedge\EE.
\end{align*}
It is straightforward to see that these cochains are in fact cocycles.
We will now show that the classes of these cocycles are linearly independent,
so that $\dim\HH^3(U)\geq4$.
We take a linear combination
$\omega=\sum_{i=1}^4\lambda_i\omega_i$ with $\lambda_1,\ldots,\lambda_4\in\kk$
and suppose that there
exists a cochain~$\xi$ in $\X^2$ such that $d^2(\xi)=\omega$. 
Since the component of~$\omega$ in~$\xx\wedge\yy\wedge\DD$ is zero, we may
write
\[
  \xi = u\otimes\xx\wedge\EE + v\otimes\yy\wedge\EE + w\otimes\DD\wedge\EE,
\]
with $u$, $v$ and $w$ in $U_1$, and there exist then
$\alpha_i,\beta_i,\gamma_i\in \kk[E]$ with $1\leq i\leq3$ such that 
\begin{align*}
  &u=x\alpha_1+ y\alpha_2 + D\alpha_3,
  &&v=x\beta_1+ y\beta_2 + D\beta_3,
  &&w=x\gamma_1+ y\gamma_2 + D\gamma_3.
\end{align*}
We now examine each component of the equality $d^2(\xi)=\omega$. In
$\xx\wedge\yy\wedge\DD$ there is nothing to see. In $\xx\wedge\yy\wedge\EE$ we
have $-[y,u]+[x,v]=\lambda_1 D^2$, or
\[
  -xy\alpha'_1- y^2\alpha'_2 - yD\alpha'_3 - F(\alpha'_3- \alpha_3)
  +x^2\beta'_1+ xy\beta'_2 + xD\beta'_3
  = \lambda_1 D^2.
\]
This is an equality in $U_2$, which we may decompose as
$\bigoplus_{i+j+k=2}x^iy^jD^k\kk[E]$. Looking at $D^2\kk[E]$ we get $\lambda_1 =0$,
from $yD\kk[E]$, $x^2\kk[E]$ and $xD\kk[E]$ we obtain $\alpha_3'=\beta'_1=\beta'_3=0$
and $xy\kk[E]$ and $y^2\kk[E]$ tell us that $\alpha_3=\alpha_2'$ and
$\beta'_2=\alpha'_1-t\alpha_3$.

In $\xx\wedge\DD\wedge\EE$, equation $d^2(\xi)=\omega$ reads
\[
  -xD\alpha_1' - F\alpha_2 - yD\alpha'_2
  +x^2 \gamma'_1 + xy\gamma_2' + xD\gamma_3'
  = \lambda_2 (D^2-2yDE + y^2(E^2-E)).
\]
The component in $D^2\kk[E]$ of this equality is $0=\lambda_2$. From
$xD\kk[E]$ and $yD\kk[E]$ we obtain $\gamma_3'=\alpha_1'$ and $\alpha_2'=0$ and
from $x^2\kk[E]$, $xy\kk[E]$ and $y^2\kk[E]$ we get $\gamma_1'=0$, $\gamma_2'=t\alpha_2$
and~$\alpha_2=0$. In particular, that $\alpha_2=0$ implies that
$\alpha_3=0$ and that 
\[\label{eq:h3bg}
  \beta'_2=\alpha'_1 = \gamma'_3
\]

We finally look at the component in $\yy\wedge\DD\wedge\EE$ of
$d^2(\xi)=\omega$, which is
\[
  tuy + (tx+2y)v - [y,v] - [D,v] + [y,w]= \lambda_3D^2 + \lambda_4xD.
\]
This is an equality in $U_2=\bigoplus_{i+j+k=2}x^iy^jD^k\kk[E]$. 
In $D^2\kk[E]$ we have $0 = \lambda_3$, and in $yD\kk[E]$
\[
  2\beta_3yD -yD\beta_2' + yD\gamma_3'=0,
\]
which in the light of~\eqref{eq:h3bg} implies $\beta_3=0$.
With this at hand we see that in $xD\kk[E]$ it only remains $0=\lambda_4$.

We have seen at this point that the only coboundary among the cocycles of the
form $\omega = \sum_{i=1}^4\lambda_i\omega_i$ is $\omega = 0$. This shows that
the classes of $\omega_1,\ldots,\omega_4$ are linearly independent, thus
finishing~the~proof.
\end{proof}

\begin{Corollary}\label{coro:HHDiffA}
Let $\A$ be a central arrangement of three lines. The Hilbert series
of~$\HH^\*(\Diff\A)$ is 
\[
  h_{\HH^\*(\Diff\A)}(t) = 1 +3t + 6t^2+4t^3.
\]
\end{Corollary}

\begin{proof}
Proposition~\ref{prop:hh3} implies at once that
the spectral sequence degenerates at $E_2$. The dimensions in the statement
are a consequence of the convergence of the sequence
in~Corollary~\ref{coro:spectral} and the information in
Proposition~\ref{prop:dim2page}.
\end{proof}

As a consequence of the information we have gathered so far we can easily
describe the Lie algebra structure on $\HH^1(\Diff\A)$.  Let us, again, call
$U=\Diff\A$ and recall that $\HH^1(U)$ is isomorphic to the space $\Out U$ of
outer derivations of $U$, that is, the quotient of the derivations of $U$
modulo inner derivations, and that 
the commutator of derivations induces a Lie
algebra structure on $\Out U$.  We know from~\cite{koma}*{Proposition 4.2}
that if $f\in S_1$ divides $xF$ then there is a derivation $\partial_f:U\to U$
such that $\partial_f(x)=\partial_f(y)=0$,
$\partial_f(D)=\frac{F}{f}\partial_yf$ and $\partial_f(E)=1$.  Let then
$f_1=x$, $f_2=y$ and $f_3=tx+y$ and put $\partial_i\coloneqq \partial_{f_i}$
for $1\leq i \leq 3$.

\begin{Corollary}\label{coro:abelian}
Let $\A$ be a central arrangement of three lines. 
The Lie algebra of outer derivations of $\Diff\A$ together with the commutator
is an abelian Lie algebra of dimension three generated by the classes of the
derivations $\partial_1$, $\partial_2$ and $\partial_3$.
\end{Corollary}

\begin{proof}

We claim that the classes of $\partial_1$, $\partial_2$ and $\partial_3$ are
linearly independent in $\Out(U)$. Indeed, let $u\in U$ and
$\lambda_1$, $\lambda_2$, $\lambda_3\in\kk$ be such that
\[\label{eq:gerst}
  \sum\lambda_i\partial_i(v)=[u,v]\qquad\text{for every $v\in U$.}
\]
Evaluating~\eqref{eq:gerst}
on each $s\in S$ the left side vanishes 
and therefore Proposition~\ref{prop:h2SU}
tells us that $u\in S$. Write $u=\sum_{j\geq0}u_j$ with $u_j\in S_j$.
Evaluating now~\eqref{eq:gerst} on $E$ we obtain
$\sum_i\lambda_i=-\sum_jju_j$. In each homogeneous component $S_j$ with
$j\neq0$ we have $ju_j=0$ and therefore $u\in S_0=\kk$ and
$\sum_i\lambda_i=0$. This equation and the one we get
evaluating~\eqref{eq:gerst}
on $D$, that is, $\lambda_2(tx+y)+\lambda_3y=0\in S_1$, finally tell us that
$\lambda_1=\lambda_2=\lambda_3=0$. 

The classes of $\partial_1$, $\partial_2$ and $\partial_3$ span $\Out U$
because, thanks to Corollary~\ref{coro:HHDiffA}, its dimension is three. 
The composition $\partial_i\circ\partial_j:U\to U$ is evidently equal to zero
for any $1\leq i,j\leq3$, as a straightforward calculation shows, and
therefore the Lie algebra structure in $\Out U$~vanishes.
\end{proof}

It is possible also to use the spectral sequence of
Corollary~\ref{coro:spectral} to obtain $\HH^\*(\Diff\A)$
for arrangements with any $l\geq3$, but we will not perform this calculation
here. The result~is
\[
h_{\HH^\*(\Diff\A)}(t) =
\begin{cases*}
1+lt + 2lt^2 +(l+1)t^3,		& if $l=3,4$; \\
1+lt + (2l-1)t^2 +lt^3, 	& if $l\geq5$.
\end{cases*}
\]
This shows that the case in which $l$ is $3$ or $4$ is \emph{genuinely}
different to that in which~$l\geq5$.
If $l\leq2$, the algebra
$\Diff\A$ is not very interesting, since it is isomorphic to algebras with
well-known Hochschild cohomology ---see~\cite{koma}*{\S3.8}.

\section{Other applications}\label{sec:applications}

\subsection{The Hochschild cohomology of a family of subalgebras of the Weyl algebra}
\label{sec:Ah}
Let $\kk$ be a field of characteristic zero, fix a nonzero $h\in \kk[x]$ and
consider the algebra $A_h$ with presentation
\[
	\dfrac{\kk\lin{ x, y}}{\left( yx-xy - h \right)}.
\]
Setting $h=1$ the algebra $A_h$
is the Weyl algebra $A_1$ that already appeared in Example~\ref{ex:weyl}, when
$h=x$ it is the universal enveloping algebra of the two-dimensional
non-abelian Lie algebra and if $h=x^2$, it is the \emph{Jordan~plane}
studied in~\cite{jordan}.

We let $S=\kk[x]$ and consider the Lie algebra $L$ freely generated by
$y=h\frac{d}{dx}$ as an $S$-submodule of $\Der S$.  It is straightforward to
see that $(S,L)$ is a Lie--Rinehart algebra whose enveloping algebra $U$ is
isomorphic to~$A_h$.  We will use the spectral sequence of
Corollary~\ref{coro:spectral} to compute the Hochschild
cohomology~$\HH^\*(A_h)$ of~$A_h$:  we will describe explicitly the second
page and find that the spectral sequence degenerates at that page.

\subsubsection{The Hochschild cohomology \texorpdfstring{$H^\*(S,U)$}{H(S,U)}}
The augmented Koszul complex
\[
  P_\* : 
  \begin{tikzcd}[column sep=1.75em] 
  0\arrow[r]
  & S^e\arrow[r,"\delta_1"]
  & S^e\arrow[r,"\varepsilon"]
  & S
  \end{tikzcd}
\]
with $\delta_1(s\otimes t) = sx\otimes t-s\otimes xt$ and augmentation
$\varepsilon(s\otimes t)=st$ is an
$S^e$-projective resolution of $S$ and therefore the Hochschild
cohomology~$H^\*(S,U)$ is, after identifying $\hom_{S^e}(S^e,U)$ with $U$, the
cohomology of the complex $U\xrightarrow{\delta} U$ with differential
$\delta(u)=[x,u]$.

\begin{Proposition}\label{prop:Ah:HSU}
There are isomorphisms of vector spaces 
\begin{align}
	&  H^0(S,U)\cong S, 
  && H^1(S,U)\cong U/hU,
  && H^q(S,U)=0\quad\text{if $q\geq2$.}
\end{align}
\end{Proposition}

\begin{proof}
The isomorphisms in the statement come, of course, of the computation of the
cohomology of $U\xrightarrow{\delta} U$. Let us first deal with $\ker\delta$.
Writing $u=\sum_{i=0}^r f_iy^i$ with $f_1,\ldots,f_r\in S$ and $r$ the
greatest index such that $f_r\neq0$, we have that
\[\label{eq:Ah:delta}
\delta(u)
= rf_rhy^{r-1} + v_u
\]
for some $v_u\in\bigoplus_{i=0}^{r-2}Sy^{i}$. If $\delta(u)=0$ then its
principal symbol $rf_rhy^{r-1}$ must be equal to zero and then, because the
field has characteristic zero, either $r=0$ or $f_r=0$. This second
possibility contradicts our assumptions, and therefore $r=0$ and $u\in S$.
That, reversely, $S$ is contained in the kernel of $\delta$ is evident.

The second claim of the statement follows from the fact that the image
of~$\delta$ is the right ideal generated by $h$, that is, $hU$. For this, we
can see that $hSy^i$ belongs to the image of $\delta$ for every $i\geq0$ with
a straightforward inductive argument using~\eqref{eq:Ah:delta}.
\end{proof}

\subsubsection{The action of \texorpdfstring{$U$ on $H^\*(S,U)$}{U on H(S,U)}}
As $S$ acts just by left multiplication,
to determine the action of $U$ it is enough to explicit that of 
$y$: for this have at hand Remark~\ref{def:alphanabla}. 

\begin{Proposition}\label{prop:Ah:action}
  Under the isomorphisms $H^0(S,U)\cong U$ and $H^1(S,U)\cong U/hU$ of
  Proposition~\ref{prop:Ah:HSU}, the action
  of $L$ on~$H^\*(S,U)$ is determined by
  \begin{align*}
  	\nabla_y^0(s) &= hs' \qquad\text{for $s\in S$;}\\
   	\nabla_y^1(\bar u ) &=-\overline{h'u} \qquad\text{for $u\in U$,}
  \end{align*}
  where the overline denotes class modulo $hU$.
\end{Proposition}

\begin{proof} 
We use Example~\ref{ex:alpha0} to see that $y$ acts on
$H^0(S,U)=S$ as in the statement. To describe its action on $H^1(S,U)$ we need a
lifting~$y_\*=(y_0,y_1)$ of $y_S:S\to S $ to $P_\*$. Let us define
\begin{align*}
&y_0(s\otimes t)=hs'\otimes 1+1\otimes ht,
&&y_1(s\otimes t)= hs'\otimes 1+1\otimes ht'+s\Delta(h)t,
\end{align*}
where $\Delta:S\to S^e$ is the unique derivation of $S$ such that
$\Delta(x)=1\otimes 1$, that is, it is the only linear map such that
$
\delta(x^j) = \sum_{s+t=j+1} x^s\otimes x^t
$
if $j\geq0$.
We readily see that
$y_0$ and $y_1$ are $y^e_S$-operators and that the diagram
\[
\begin{tikzcd}[column sep=1.75em] 
S^e\arrow[r,"\delta_1"]
& S^e\arrow[r,"\varepsilon"]
& S \\
S^e\arrow[r,"\delta_1"]\arrow[u,"y_1"]
& S^e\arrow[r,"\varepsilon"]\arrow[u,"y_0"]
& S\arrow[u,"y"]
\end{tikzcd}
\]
commutes, and thus the pair $(y_0,y_1)$ is in fact one of the
$y^e_S$-liftings we were looking for.

We now compute $y_1^\sharp: \hom_{S^e}(S^e,U)\to \hom_{S^e}(S^e,U)$ following
\eqref{eq:alphasharp}, and for that we let $\phi\in\hom_{S^e}(S^e,U)$.
Bearing in mind the isomorphism $\hom_{S^e}(S^e,U)\cong U$ induced by the
evaluation in $1\otimes 1$, we need only compute 
\begin{align*}
y_1^\sharp(\phi)(1\otimes 1) 
&= [y,\phi(1\otimes1)] - \phi(\Delta(h)).
\end{align*}
Assuming without losing generality that $\phi(1\otimes 1)= fy^i$ and
$h=x^j$ for~$f\in S $ and $i$,~$j\geq0$, we obtain
\begin{align*}
  y_1^\sharp(\phi)(1\otimes1) 
    &= [y,fy^i] - \sum_{s+t=j+1} x^sfy^i x^t 
     = hf'y^i - \sum_{s+t=j+1} x^sf(x^ty^i + [y^i,x^t]) \\
    &\equiv - jx^{j-1}f 
      \mod hU,
\end{align*} 
since~$[y^i,x^t]\in h U$ for all $i,t\geq0$. 
Taking class in cohomology and identifying $\phi$ with $\phi(1\otimes 1)$,
we get
\[
\nabla_y^1(\overline{fy^i})
= -\overline{h'fy^i},
\]
and the stated result follows from this.
\end{proof}

\subsubsection{The Lie-Rinehart cohomology}
Let us now compute $H_S^\*(L,H^i(S,U))$ for each $i\in\ZZ$. Using the complex
in Proposition~\ref{prop:ch-ei} to compute Lie--Rinehart cohomology of $S$, we
see that this is the cohomology of the complex 
\[
\begin{tikzcd}[column sep=1.75em] 
  H^i(S,U)\arrow[r,"\nabla_y^i"]
  &H^i(S,U).
\end{tikzcd}
\]

\begin{Proposition}\label{prop:Ah2nd}
Let $d=\gcd(h,h')$ and let $I$ be the ideal of $S/(h)$ generated by the
class of $h/d$. There are isomorphisms of vector spaces
\[\label{eq:Ah2nd}
  \begin{aligned}
  &H^0_S(L,H^0(S,U))\cong \kk,
  &&H^1_S(L,H^0(S,U))\cong S/(h), 
\\
  &H^0_S(L,H^1(S,U))\cong I[y],
  &&H^1_S(L,H^1(S,U))\cong S/(d)[y].
  \end{aligned}
\]
and $H^p_S(L,H^q(S,U))=0$ if $p,q\geq2$.
\end{Proposition}

\begin{proof}
We make use of the explicit description of $\nabla_y^i$ in
Proposition~\ref{prop:Ah:action}.  For $i=0$, this amounts to the cohomology
of $S\xrightarrow{y} S$, and we readily see that the kernel of this map is
$\kk$ and its image, $hS$.  

Consider now the case in which $i=1$ and recall that $H^1(S,U)$ is isomorphic
to $U/hU$. As $U/hU$ is the quotient of the free noncommutative algebra
in $x$ and $y$ by the relations $xy-yx=h$ and $h=0$, we may identify
$H^1(S,U)$ with $\frac{S}{(h)}[y]$. 

For each $f\in S$ we write $\tilde f$ its class in $S/(h)$. This way, 
given $v\in \frac{S}{(h)}[y]$,
there are $f_0,\dots,f_r\in\kk[x]$ such that $v=\sum_{i=0}^r\tilde
f_iy^i$ and, as our findings on $\nabla_y^1$ of
Proposition~\ref{prop:Ah:action}
allow us to see,
\[\label{eq:y1}
\nabla_y^1(v)
=-\sum_{i=0}^r\widetilde{h' f_i}y^i.
\]

It is immediate that the cokernel of 
$\nabla_y^1$ is $\frac{S}{(h,h')}[y]$. To compute its kernel, let us
suppose that~$\nabla_y^1(u)=0$. For each $i\in\{0,\ldots,r\}$
we have that $h$ divides $h'f_i$
and therefore, if $d$ denotes the greatest common divisor of $h$ and $h'$, 
we have that  $h/d$ divides $f_i$. Denoting by $I$ the ideal
of $\frac{S}{(h)}$ generated by the class of $h/d$, we conclude that
$H_S^0(L,H^1(S,U))$
is isomorphic to the space of polynomials $I[y]$ with coefficients in $I$.
\end{proof}

The conclusion of this calculation is the following description of the
Hochschild cohomology of $U$ ---this time we write $A_h$ instead~of~$U$.

\begin{Proposition}
There are isomorphisms of vector spaces
\[
	\HH^i(A_h) \cong
		\begin{cases*}
	\kk 	& if $i=0$; \\
S/(h)\oplus I[y]
& if $i=1$;\\
\dfrac{S}{(d)}[y] 
& if $i=2$; \\
		0		& otherwise,
\end{cases*}
\]
where $d$ stands for the greatest common divisor of $h$ and its derivative~$h'$
and $I$ is the ideal of $\frac{S}{(h)}$ generated by the class of $h/d$.
\end{Proposition}


\begin{proof}
The spectral sequence $E_\*$ of
Corollary~\ref{coro:spectral} converges to $\HH^\*(A_h)$ and its second page,
given by $E_2^{p,q}=H_S^p(L,H^q(S,U))$, is completely computed
in Proposition~\ref{prop:Ah2nd}.  The sequence degenerates because
$E_2^{p,q}=0$ if $p,q\geq2$.
\end{proof}

The first cohomology space had already been obtained, in other words,
in~\cite{ah2}*{Theorem~5.7.\itshape(iii)} and the second one
in~\cite{lopes-solotar}*{Corollary~3.11}.  
However, the spectral sequence argument provides conceptual
simplifications and, as a consequence of
that, this computation is significantly shorter.

\subsection{The Van den Bergh duality property for~\texorpdfstring{$U$}{U}}
\label{sec:U}
An appropriate 
specialization of Corollary~\ref{coro:spectral} allows us to recover one of
the main results of~\cite{th-pa}, which we recall after the following
preliminary definition. Let $n\geq0$.  An algebra~$A$ has \emph{Van den Bergh
duality} of dimension $n$ if~$A$ has a resolution of finite length by
finitely generated projective $A$-bimodules and there exists an invertible
$A$-bimodule $D$ such that there is an isomorphism of $A$-bimodules
\[
    \Ext_{A^e}^i(A,A\otimes A) =
    \begin{cases*}
        0       & if $i\neq n$; \\
        D       & if $i=n$.
    \end{cases*}
\]
The Van den Bergh duality property for an algebra~$A$ is important because, as
can be seen in~\cite{vanden}, it relates the Hochschild cohomology of $A$ with
its homology in a way analogue to Poincar\'e duality: indeed, for each
$A$-bimodule $M$ it produces a canonical isomorphism $H^i(A,M) \to
H_{n-i}(A,D\otimes_A M)$.

Let us consider the left $U$-module structure
on~$\Lambda_S^dL^\vee\otimes_SD$ discussed in the first section
of~\cite{th-pa}. If $F$ is the functor from left
$U$-modules to $U^e$-modules in~\eqref{eq:Ffunctor} then evidently 
$F(\Lambda_S^dL^\vee\otimes_SD)$ becomes an $U^e$-module.

\begin{Theorem}
\label{thm:th-pa}
Let $(S,L)$ be a Lie--Rinehart algebra such that $S$ has Van den Bergh duality in
dimension $n$ and $L$ is finitely generated and projective with constant rank
$d$ as an $S$-module and let $L^\vee = \hom_S(L,S)$. The enveloping algebra~$U$
of the algebra has Van den Bergh duality in dimension $n+d$ and there is an
isomorphism of $U^e$-modules 
\[
	\Ext_{U^e}^{n+d} (U, U^e)	
		\cong F(\Lambda_S^dL^\vee\otimes_SD).
\]
\end{Theorem}

\begin{Lemma}\label{lem:th-pa}
Let $A$ be an algebra and $T$ and $P$ two $A$-modules such that $T$ admits a
projective resolution by finitely generated $A$-modules and $P$ is flat.
There is an isomorphism
\[
	\Ext_A^\* (T,P)
		\cong \Ext_A^\*(T,A)\otimes_A P.
\]
\end{Lemma}

\begin{proof}[Proof of Lemma~\ref{lem:th-pa}]
Let $Q_\*$ be such a resolution of $T$. For each $i\geq0$, the evident map
from $ \hom_A(Q_i,A)\otimes_A P$ to $\hom_A (Q_i,P)$ is an isomorphism
because $Q_i$ is finitely
generated and projective. As $P$ is flat, 
the cohomology of the complex $ \hom_A(Q_\*,A)\otimes_A P $ is isomorphic to
$\Ext_A^\*(T,A)\otimes_A P$.
\end{proof}

\begin{proof}[Proof of Theorem~\ref{thm:th-pa}]
The homological smoothness of $U$ follows
from Lemma~5.1.2 of \cite{th-pa}, whose proof does not depend on this theorem.

Let us write $D$
for the dualizing bimodule $\Ext^n_{S^e}(S,S^e)$.
We take, specializing Corollary~\ref{coro:spectral}, $M=U^e$
to obtain a spectral sequence $E_\*$ such that
\[
E_2^{p,q}
	= H_S^p(L, H^q(S,U^e) )
	\implies H^{p+q} (U, U^e).
\]
Let us first deal with $H^q(S,U^e)$. 
%
As we observed in the proof of Proposition~\ref{prop:injectives},
the $U^e$-module
$U^e$ is $S^e$-projective and, since $S$ has Van den Bergh duality, it
admits a resolution by finitely generated projective $S^e$-modules.
We may therefore
use Lemma~\ref{lem:th-pa} to see that 
\[
H^q(S,U^e)
	\cong H^q(S,S^e)\otimes_{S^e} U^e,
\]
which is zero if $q\neq n$ and isomorphic to $D\otimes_{S^e}U^e$ if
$q=n$.
As a consequence of this, our spectral sequence $E_\*$ 
degenerates at its second page and thus $H^{p+n}(U,U^e)$ is
isomorphic to
$H_S^p(L,  D\otimes_{S^e} U^e)$ for each $p\in\ZZ$.

The Chevalley--Eilenberg complex from Proposition~\ref{prop:ch-ei} is an
$U$-projective resolution of $S$ by finitely generated modules. On the other
hand, the dualizing
module $D$ is
$S$-projective because it is invertible ---see Chapter 6 in the book~\cite{AF}
by F\@. Anderson and K\@. Fuller. We conclude that
the $U$-module $D\otimes_{S^e} U^e$
is projective, and we can apply
Lemma~\ref{lem:th-pa} we obtain an isomorphism
\[
 H_S^\*(L, D\otimes_{S^e} U^e)
	\cong H_S^\*(L,U) \otimes_U( D\otimes_{S^e} U^e).
\]
Now, the hypotheses on $L$ are such that
Theorem 2.10 in~\cite{hueb:duality} tells us that
$H_S^p(L,U)$ is zero if $p\neq d$ and is isomorphic to $\Lambda_S^dL^\vee$ if 
$p=d$, so that
\[
H^i(U,U^e)
	\cong \begin{cases*}
		\Lambda_S^dL^\vee\otimes_U( D\otimes_{S^e} U^e),
			& if $i=n+d$;\\
		0	& otherwise.
		\end{cases*}
\]
The dualizing bimodule of $U$ is therefore isomorphic to
$\Lambda_S^dL^\vee\otimes_U( D\otimes_{S^e} U^e)$, or, 
as an immediate application of Lemma 3.5.2 in~\cite{th-pa} shows,
to $F(\Lambda_S^dL^\vee\otimes_SD)$.
\end{proof}


\end{document}